\documentclass[12pt,reqno]{amsart}
\usepackage{etex}
\usepackage{dsfont,amsmath,amsfonts,amscd,amssymb,graphicx,mathrsfs,eufrak,upgreek}
\usepackage[dvips,all,arc,curve,color,frame]{xy}
\usepackage[usenames]{color}
\usepackage{setspace}
\usepackage{array,multirow,booktabs,longtable}

\makeatletter
\newcommand{\customlabel}[3]{#3\def\@currentlabel{#2}\label{#1}}
\makeatother

\usepackage[colorlinks,hyperfootnotes=false]{hyperref}
\newcommand{\hrefsf}[2]{\href{#1}{\textsf{#2}}}
\newcommand{\hrefsfopt}[3]{{#1}\hrefsf{#2}{#3}}

\newcommand{\marginparstretch}{0.6}
\let\oldmarginpar\marginpar
\renewcommand\marginpar[1]{\-\oldmarginpar[\framebox{\setstretch{\marginparstretch}\begin{minipage}{\marginparwidth}{\raggedleft\tiny #1}\end{minipage}}]{\framebox{\setstretch{\marginparstretch}\begin{minipage}{\marginparwidth}{\raggedright\tiny #1}\end{minipage}}}}

\usepackage[colorlinks]{hyperref}
\usepackage{tikz,tikz-cd,mathrsfs}
\usetikzlibrary{arrows,positioning,calc,intersections}

\usepackage{pgfplots}
\pgfplotsset{compat=1.8}
\usetikzlibrary{decorations.markings}

\tikzset{
  mid arrow/.style={postaction={decorate,decoration={
        markings,
        mark=at position .5 with {\arrow{stealth}}
      }}},
}

\tikzset{
        Point/.style={circle,draw=black,circle,fill=black,inner sep=0pt, minimum size=2pt},
        DynkinBlack/.style={circle,draw=black,circle,fill=black,inner sep=0pt, minimum size=4pt},
         DynkinWhite/.style={circle,draw=black,circle,fill=white,inner sep=0pt, minimum size=4pt},
        vertex/.style={circle,fill=black,inner sep=1pt,outer sep=8pt},
        star/.style={circle,fill=yellow,inner sep=0.75pt,outer sep=0.75pt},
        gap/.style={inner sep=0.5pt,fill=white}}

\tikzstyle{mybox} = [draw=black, fill=blue!10, very thick,
    rectangle, rounded corners, inner sep=10pt, inner ysep=20pt]
\tikzstyle{boxtitle} =[fill=blue!50, text=white,rectangle,rounded corners]

\newcommand{\arrow}[2][20]
 {
  \hspace{-5pt}
  \begin{tikzpicture}
   \node (A) at (0,0) {};
   \node (B) at (#1pt,0) {};
   \draw [#2] (A) -- (B);
  \end{tikzpicture}
  \hspace{-5pt}
 }

\newcommand{\arrowsplit}{0.4ex}

\newcounter{envcount} \numberwithin{envcount}{section}
\newcounter{keyenvcount}

\newtheorem{theorem}[envcount]{Theorem}
\newtheorem{corollary}[envcount]{Corollary}

\newtheorem{keytheorem}[keyenvcount]{Theorem}

\newtheorem{keyproposition}[keyenvcount]{Proposition}

\newtheorem{proposition}[envcount]{Proposition}
\newtheorem{lemma}[envcount]{Lemma}
\newtheorem{definition}[envcount]{Definition}
\newtheorem{assumption}[envcount]{Assumption}

\theoremstyle{definition} 

\newtheorem{example}[envcount]{Example}
\newtheorem{remark}[envcount]{Remark}
\newtheorem*{keyremark}{Remark}
\newtheorem*{question}{Question}
\newtheorem{notation}[envcount]{Notation}
\newtheorem{setup}[envcount]{Setup}

\theoremstyle{remark}
\newtheorem*{acknowledgements}{Acknowledgements}
\newtheorem*{conventions}{Conventions}

\newcommand{\defined}[1]{\emph{#1}}

\def\D{\mathop{\rm{D}^{}}\nolimits}
\def\F{\mathop{\rm{F}^{}}\nolimits}
\newcommand\Dres[2]{\D(#1|#2)}

\def\flop{{\sf{F}}}

\def\Id{\mathop{\sf{Id}}\nolimits}
\def\Idmatrix{\mathds{1}}

\def\MA{M_{\rm{A}}}
\def\cMA{\overline{M}_{\!\rm{A}}}
\def\MB{M_{\kern 0.5pt \rm{B}}}
\def\cMB{\overline{M}_{\kern -0.5pt \rm{B}}}
\def\QA{\cQ_{\rm{A}}}
\def\QB{\cQ_{\kern 0.5pt \rm{B}}}

\newcommand{\cD}{\mathcal{D}}

\newcommand{\cL}{\mathcal{L}}
\newcommand{\cM}{\mathcal{M}}
\newcommand{\cN}{\mathcal{N}}
\newcommand{\cO}{\mathcal{O}}
\newcommand{\cP}{\mathcal{P}}
\newcommand{\cQ}{\mathcal{Q}}

\newcommand{\cS}{\mathcal{S}}

\newcommand\compose{\circ}
\newcommand\placeholder{-}

\newlength\tempWidth

\newcommand\KS{Kapranov--Schechtman}
\newcommand\KandS{Kapranov and Schechtman}
\newcommand\linzn{\cM}
\newcommand\sphFun{\mathsf{S}}
\newcommand\twistFun{\mathsf{T}}
\newcommand\cotwistFun{\mathsf{C}}
\newcommand\Ladj{\mathsf{LA}}
\newcommand\Radj{\mathsf{RA}}

\newcommand\labelpos{1.3}
\newcommand\labelposy{1.1}
\newcommand\labelperiod{2.8}
\newcommand\picardspace{0.5}
\newcommand\labelposrelax{0.2}
\newcommand\axisposrelax{0.4}

\newcommand\labeladjust{0.05}
\newcommand\grayCol{gray!50}

\newcommand{\axes}[4]{
 \draw[->, \grayCol] (-2.3-#2+#3,0) to (2.4+#2,0); 
 \draw[->, \grayCol] (0,-1.7-#1) to (0,1.8+#1+#4); 
}

\newcommand{\pullbackPic}[6]{
\begin{tikzpicture}[node distance=1cm, auto, line width=0.5pt]

\axes{0.5}{2.8}{0}{0}

 \node (pos) at (\labeladjust-\labelperiod,-\labelposy-\labelposrelax)  {#1};
 \node (pos) at (\labeladjust,-\labelposy-\labelposrelax)  {#1};
 \node (pos) at (\labeladjust+\labelperiod,-\labelposy-\labelposrelax)  {#1};

 \node (neg) at (\labeladjust-\labelperiod,+\labelposy+\labelposrelax) {#2};
 \node (neg) at (\labeladjust,+\labelposy+\labelposrelax) {#2};
 \node (neg) at (\labeladjust+\labelperiod,+\labelposy+\labelposrelax) {#2};

 \draw[<-, rounded corners=15pt]  (-1.85-\picardspace,-1.8) -- (-1-\picardspace,-2) -- (-0.35-\picardspace,-1.8) ;
 \node (b) at (-1-\picardspace,-2.2) {\scriptsize #3};

 \draw[->, rounded corners=15pt]  (1.85+\picardspace,-1.8) -- (1+\picardspace,-2) -- (0.35+\picardspace,-1.8) ;
 \node (b) at (+1+\picardspace,-2.2) {\scriptsize #3};

 \draw[->, rounded corners=15pt]  (-1.85-\picardspace,1.8) -- (-1-\picardspace,2) -- (-0.35-\picardspace,1.8) ;
 \node (c) at (-1-\picardspace,2.2) {\scriptsize #4};

 \draw[<-, rounded corners=15pt]  (1.85+\picardspace,1.8) -- (1+\picardspace,2) -- (0.35+\picardspace,1.8) ;
 \node (c) at (+1+\picardspace,2.2) {\scriptsize #4};

 \draw[->, rounded corners=15pt]  (-\labelperiod-0.3,-0.85) -- (-\labelperiod-0.5,0) -- (-\labelperiod-0.3,0.85) ;
 \node (b) at (-\labelperiod-0.7,0) {\scriptsize #5};
 \draw[<-, rounded corners=15pt]  (-\labelperiod+0.3,-0.85) -- (-\labelperiod+0.5,0) -- (-\labelperiod+0.3,0.85) ;
 \node (c) at (-\labelperiod+0.7,0) {\scriptsize #6};

 \draw[->, rounded corners=15pt]  (-0.3,-0.85) -- (-0.5,0) -- (-0.3,0.85) ;
 \node (b) at (-0.7,0) {\scriptsize #5};
 \draw[<-, rounded corners=15pt]  (0.3,-0.85) -- (0.5,0) -- (0.3,0.85) ;
 \node (c) at (+0.7,0) {\scriptsize #6};

 \draw[->, rounded corners=15pt]  (+\labelperiod-0.3,-0.85) -- (+\labelperiod-0.5,0) -- (+\labelperiod-0.3,0.85) ;
 \node (b) at (+\labelperiod-0.7,0) {\scriptsize #5};
 \draw[<-, rounded corners=15pt]  (+\labelperiod+0.3,-0.85) -- (+\labelperiod+0.5,0) -- (+\labelperiod+0.3,0.85) ;
 \node (c) at (+\labelperiod+0.7,0) {\scriptsize #6};

\node (dots) at (-\labelperiod-1.7, 1.35) {$\dots$};
\node (dots) at (+\labelperiod+1.7, 1.35) {$\dots$};
\node (dots) at (-\labelperiod-1.7, -1.35) {$\dots$};
\node (dots) at (+\labelperiod+1.7, -1.35) {$\dots$};

\draw (\labelperiod,0) circle (.4ex);
\draw (-\labelperiod,0) circle (.4ex);
\draw (0,0) circle (.4ex);

\end{tikzpicture}
}

\newcommand{\monodromiesCentralCharge}[4]{
\begin{tikzpicture}[node distance=1cm, auto, line width=0.5pt]

\axes{-0.3}{-0.35}{0}{0.4}

 \node (pos) at (0,1.2)  {#1};

 \draw[<-, rounded corners=10pt] (0.1,0.6) -- (0.35,-0.4) -- (-0.35,-0.4) -- (-0.1,0.6);
 \draw[<-, rounded corners=10pt]  (0.4,0.85) -- (1.5,0) -- (1,-0.5) -- (0.3,0.7) ;
 \draw[->, rounded corners=10pt]  (-0.4,0.85) -- (-1.5,0) -- (-1,-0.5) -- (-0.3,0.7) ;

\node at (-1.5,-0.8) {\scriptsize #2};
\node at (0,-0.9) {\scriptsize #3};
\node at (1.5,-0.8) {\scriptsize #4};

\node (dots) at (-1.5, 0.5) {$\dots$};
\node (dots) at (1.6, 0.5) {$\dots$};

\draw[fill=black] (0,0) circle (.2ex);
\draw[fill=black] (+1,0) circle (.2ex);
\draw[fill=black] (-1,0) circle (.2ex);

\end{tikzpicture}
}

\newcommand{\monodromiesStandardFlop}[4]{
\begin{tikzpicture}[node distance=1cm, auto, line width=0.5pt]

\axes{0.3}{\axisposrelax}{0.8}{-0.1}

 \node (neg) at (+\labelpos+\labelposrelax,-\labeladjust) {#1};

 \draw[<-, rounded corners=10pt] (0.6,-0.1) -- (-0.4,-0.35) -- (-0.4,0.35) -- (0.6,0.1);
 \draw[<-, rounded corners=10pt]  (0.85,-0.4) -- (0,-1.5) -- (-0.5,-1) -- (0.7,-0.3) ;
 \draw[->, rounded corners=10pt]  (0.85,0.4) -- (0,1.5) -- (-0.5,1) -- (0.7,0.3) ;

\draw (0,-1) circle (.4ex);
\draw (0,0) circle (.4ex);
\draw (0,+1) circle (.4ex);

\node at (-0.8,-1.2) {\scriptsize #2};
\node at (-0.8,0) {\scriptsize #3};
\node at (-0.7,1.2) {\scriptsize #4};

\node (dots) at (0.5, -1.3) {$\vdots$};
\node (dots) at (0.5, 1.5) {$\vdots$};

\end{tikzpicture}
}

\newcommand{\halfMonodromiesStandardFlop}[3]{
\begin{tikzpicture}[node distance=1cm, auto, line width=0.5pt]

\axes{0.3}{\axisposrelax}{0}{-0.1}

 \node (pos) at (-\labelpos-\labelposrelax,-\labeladjust) {#2};
 \node (neg) at (+\labelpos+\labelposrelax,-\labeladjust) {#1};

 \draw[<->, rounded corners=30pt]  (-1.1,-0.4) -- (0,-1.8) -- (1.1,-0.4) ;
 \node (b) at (1.0,-1.2) {\scriptsize $#3^{-1}$};
 \draw[<->, rounded corners=15pt]  (-0.85,-0.3) -- (0,-0.8) -- (0.85,-0.3) ;
 \node (b) at (0.4,-0.8) {\scriptsize $#3^0$};
 \draw[<->, rounded corners=15pt]  (-0.85,0.3) -- (0,0.7) -- (0.85,0.3) ;
 \node (b) at (0.4,0.8) {\scriptsize $#3^1$};
 \draw[<->, rounded corners=30pt]  (-1.1,0.4) -- (0,1.8) -- (1.1,0.4) ;
 \node (b) at (0.8,1.3) {\scriptsize $#3^2$};

\draw (0,-1) circle (.4ex);
\draw (0,0) circle (.4ex);
\draw (0,+1) circle (.4ex);

\node (dots) at (-0.6, -1.4) {$\vdots$};
\node (dots) at (-0.6, 1.6) {$\vdots$};

\end{tikzpicture}
}

\newcommand{\monodromiesSKMSshared}[8]{

\axes{0}{2.8}{0}{0}

 \draw[->,looseness=10] (2.6,-0.3) to[out=-45, in=45] (2.6,0.3);
 \draw[->,looseness=10] (-2.6,0.3) to[out=180-45, in=180+45] (-2.6,-0.3);

\node at (4.3,0) {\scriptsize #3};
\node at (-4.3,0) {\scriptsize #4};

\draw (3.3,0) circle (.4ex);
\draw (-3.3,0) circle (.4ex);

\ifnum #7=0
  \node at (3.3,-0.3) {\scriptsize $+#6$};
  \node at (-3.3,-0.3) {\scriptsize $-#6$};
  \draw[fill=black] (0,0) circle (.2ex);
\else
  \node at (0,-0.3) {\scriptsize #8};
  \draw (0,0) circle (.4ex);
\fi

}

\newcommand{\monodromiesSKMS}[8]{
\begin{tikzpicture}[node distance=1cm, auto, line width=0.5pt]

\monodromiesSKMSshared{#1}{#2}{#3}{#4}{#5}{#6}{#7}{#8}

 \draw[->,looseness=1]  (0.85,-0.5) to[out=-90-45, in=-45] (-0.85,-0.5) ;
 \draw[->,looseness=1] (-0.85,0.4)  to[out=45, in=90+45] (0.85,0.4) ;

 \node (b) at (0,-1.1) {\scriptsize #5};
 \node (c) at (0,1.0) {\scriptsize #5};

 \node (neg) at (+\labelpos+\labelposrelax,-\labeladjust) {#1};
 \node (pos) at (-\labelpos-\labelposrelax,-\labeladjust) {#2};

\end{tikzpicture}
}

\newcommand{\monodromiesSKMSpair}[7]{
\begin{tikzpicture}[node distance=1cm, auto, line width=0.5pt]

\monodromiesSKMSshared{#1}{#2}{#3}{#4}{#5}{#6}{0}{}

\draw (0,0) ellipse (2.2 and 0.8);

\draw[->] (-2.4, 0.3) arc (180-5:5:2.4 and 0.8);
\draw[<-] (-2.4, -0.3) arc (180-5:5:2.4 and -0.8);

 \node (b) at (0,-1.3) {\scriptsize #5};
 \node (c) at (0,1.3) {\scriptsize #5};

 \node (neg) at (+1.9,-0.15) {\scriptsize #1};
 \node (pos) at (-1.8,-0.15) {\scriptsize #2};
\draw[fill=black] (2.2,0) circle (.2ex);
\draw[fill=black] (-2.2,0) circle (.2ex);

\draw[fill=black] (0,0) circle (.2ex);

 \node (D) at (-1.2,0.4) {#7};

\end{tikzpicture}
}

\newcommand{\skeletonForGIT}[4]{
\begin{tikzpicture}[node distance=1cm, auto, line width=0.5pt]

\axes{0}{0}{0}{0}

\draw[fill=black] (0,-1) circle (.2ex);
\draw[fill=black] (0,1) circle (.2ex);
\draw[fill=black] (0,0) circle (.2ex);
\draw[fill=black] (+1,0) circle (.2ex);

\draw (0.8,0) to (1,0);
\draw (0,0) to  (0.8,0);
\draw (0,-1) to [out=60,in=180] (0.8,0);
\draw (0,1) to [out=-50,in=180] (0.8,0);

\node (dots) at (0.2, -1.3) {$\vdots$};
\node (dots) at (0.2, 1.5) {$\vdots$};

 \node (label) at (0.7,-0.6) {#1};

 \node at (+1.3,0.2) {\scriptsize #2};
 \node at (-0.3,1) {\scriptsize #3};
 \node at (-0.3,-1) {\scriptsize #4};

\end{tikzpicture}
}

\newcommand{\skeletonForDiskShared}[7]{

\ifnum #7=0
  \draw[line width=0.5pt] (0,0) circle (1.5);
\else
  \draw[\grayCol,line width=0.5pt] (0,0) circle (1.5);
\fi

\draw[color=#4] (1.2,0) to (1.5,0);
\draw[color=#4] (-0.1,0.8) to [out=-45,in=180-40] (0.15,0.55);
\draw[color=#4] (0.15,0.55) to [out=-40,in=180] (1.2,0);
\draw[color=#4] (0,-1) to [out=60,in=180] (1.2,0);

\ifnum #5=0
  \draw[color=#4] (-0.2,0.3) to [out=-20,in=180] (1.2,0) ;
\else
   \draw[color=#4,looseness=1] (-0.2,0.3) to [out=180,in=-90] (-0.7,0.8) to [out=90,in=180] (-0.2,1.2) to [out=0,in=90+20] (0.3,0.7) to [out=-90+20,in=180] (1.2,0);
\fi

\draw[fill=black] (+1.5,0) circle (.2ex);
\draw[fill=black] (-0.1,0.8) circle (.2ex);
\draw[fill=black] (-0.2,0.3) circle (.2ex);
\draw[fill=black] (0,-1) circle (.2ex);

 \node (neg) at (+1.7,0.15) {\scriptsize #2};

  \node (label) at (0.7,-0.6) {#3};

\ifnum #6=0
 \node (neg) at (-0.35,0.8) {\scriptsize ${#1}_{n}$};
 \node (neg) at (-0.3,0.1) {\scriptsize ${#1}_{n-1}$};
 \node (neg) at (-0.25,-1) {\scriptsize ${#1}_{1}$};
\else
 \node (neg) at (-0.45,0.7) {\scriptsize ${#1}_{n}$};
 \node (neg) at (-0.35,0.1) {\scriptsize ${#1}_{n-1}$};
 \node (neg) at (-0.25,-1) {\scriptsize ${#1}_{1}$};
\fi
\node (dots) at (-0.1, -0.3) {$\vdots$};

}

\newcommand{\skeletonForDisk}[3]{
\begin{tikzpicture}[line width=0.5pt]

\skeletonForDiskShared{#1}{#2}{#3}{black}{0}{0}{0}

\end{tikzpicture}
}

\newcommand{\skeletonForDiskTwisted}[3]{
\begin{tikzpicture}[line width=0.5pt]

\skeletonForDiskShared{#1}{#2}{#3}{black}{1}{0}{0}

\end{tikzpicture}
}

\newcommand{\loopForSkeleton}[4]{
\begin{tikzpicture}[node distance=1cm, auto, line width=0.5pt]

\skeletonForDiskShared{#1}{#2}{#3}{\grayCol}{0}{1}{0}

\draw (1.2,0) to (1.5,0);
\draw (0.15,0.55) to [out=180-40,in=180+90] (-0.3,0.8) to [out=90,in=180+10] (-0.15,1.0) to [out=10,in=180-90] (0.1,0.75) to [out=-90,in=180-40] (0.15,0.55) to [out=-40,in=180] (1.2,0);

\draw[->] (-0.15,1.0) -- (-0.16,1.0);
\node (loop) at (0.15,1.15) {\scriptsize #4};

\end{tikzpicture}
}

\newcommand{\skeletonForSurfaceShared}[7]{

\skeletonForDiskShared{#1}{#2}{#3}{black}{0}{0}{#7}
\draw (-2.1,-0.2) to [out=90,in=0+180] (0,2.1) to [out=0,in=0+180] (3.2,0.5) to [out=0,in=0+180] (5,1) to [out=0,in=-90+180] (6,0) to [out=-90,in=180+180] (2,-2.25) to [out=180,in=90+180] (-2.1,-0.2);

\ifnum #7=0
  \node (D) at (-1.25,1.3) {#5};
\fi

\node (S) at (-1.75,1.75) {#6};

\begin{scope}[transform canvas={xshift=2.85cm,yshift=-0.35cm}]
\draw[opacity=0,name path=mid] (0,-0.1) -- (2.5,-0.1);
\draw[name path=lower] (0,0) to [out=-30,in=-180+30] (2,0);
\draw[name intersections={of=lower and mid}](intersection-1) to [out=30,in=-180-30] (intersection-2);

\draw[fill=black] (2.25,-0.5) circle (.2ex);
 \node (y) at (2.25+0.2,-0.5+0.15) {\scriptsize #4};
 
\end{scope};
}

\newcommand{\skeletonForSurface}[6]{
\begin{tikzpicture}[line width=0.5pt]

\skeletonForSurfaceShared{#1}{#2}{#3}{#4}{#5}{#6}{0}

\end{tikzpicture}
}

\newcommand{\skeletonForPuncturedSurface}[9]{
\begin{tikzpicture}[line width=0.5pt]

\skeletonForSurfaceShared{#1}{#2}{#3}{#4}{#5}{#6}{1}

\draw[fill=black] (2,-1.5) circle (.2ex);
 \node (y) at (2+0.2,-1.5+0.15) {\scriptsize #7};

\draw[dotted] (2,-1.5) to [out=90,in=-90] (1.5,0);
 \node (y) at (1.85,-0.4) {#8};

\draw (-1.6,0) to [out=90,in=0+180] (0,1.6) to [out=0,in=180-40] (1.6,0.8) to [out=-40,in=-80+180] (2.2,-0.4) to [out=-80,in=-90+180] (2.5,-1.4) to [out=-90,in=180+180] (2.1,-1.8) to [out=180,in=180+170] (1.5,-1.3) to [out=170,in=180+180] (0,-1.6) to [out=180,in=90+180] (-1.6,0);

 \node (y) at (1.4,-1.6) {#9};

\draw (2,-1.5) to [out=90,in=180] (1.2,0);

\end{tikzpicture}
}

\newcommand{\skeletonForPair}[4]{
\begin{tikzpicture}[node distance=1cm, auto, line width=0.5pt]

\draw[line width=0.5pt] (0,0) circle (1.5);

\draw[fill=black] (0.1,-0.2) circle (.2ex);
\draw[fill=black] (170:1.5) circle (.2ex);
\draw[fill=black] (-5:1.5) circle (.2ex);

\draw (0.1,-0.2) to [out=20,in=160] (-5:1.5);
\draw (0.1,-0.2) to [out=20-180,in=-30] (170:1.5);

 \node (label) at (0.8,-0.4) {#3};

 \node (pos) at (-1.1,0.35)  {\scriptsize #1};
 \node (neg) at (+1.9,-0.1) {\scriptsize #2};
 \node (base) at (0.1,0.1) {\scriptsize #4};

\end{tikzpicture}
}

\begin{document}

\title[Perverse schobers: constructions and examples]{\phantom{text}\vspace{-0.8cm}Perverse schobers on Riemann surfaces: \\ constructions and examples}
\author{W.\ Donovan}
\address{Yau Mathematical Sciences Center, Tsinghua University, Haidian District, Beijing 100084, China}
\email{donovan@mail.tsinghua.edu.cn}
\begin{abstract}
This note studies perverse sheaves of categories, or schobers, on Riemann surfaces, following ideas of \KandS{}~\cite{KS2}. For certain wall crossings in geometric invariant theory, I construct a schober on the complex plane, singular at each imaginary integer. I use this to obtain schobers for standard flops: in the $3$-fold case, I relate these to a further schober on a partial compactification of a stringy K\"ahler moduli space, and suggest an application to mirror symmetry.
\end{abstract}

\thanks{The author is supported by World Premier International Research Center Initiative (WPI Initiative), MEXT, Japan, and JSPS KAKENHI Grant Number~JP16K17561.}
\maketitle
\parindent 20pt
\parskip 0pt

\tableofcontents

\section{Introduction}
\label{section intro}

Perverse sheaves are key objects in the interaction of algebraic geometry, analysis, and topology~\cite{Kashiwara1,BBD1982}. They are related to D-modules by the Riemann--Hilbert correspondence~\cite{Kashiwara2,Mebkhout}, and thence control global behaviour for differential equations.

Perverse sheaves may be defined in terms of constructible sheaves, but often admit characterizations using quiver representations, in particular by work of Beilinson~\cite{Beilinson1984}, and Gelfand--MacPherson--Vilonen~\cite{GMV}.  \KandS{} have used such descriptions to define perverse sheaves of (triangulated) categories, or schobers~\cite{KS2}: in this approach, the vector spaces in a quiver representation are replaced by categories, and the linear maps by functors satisfying appropriate conditions.

Harder and Katzarkov recently used schobers to given new proofs for cases of homological mirror symmetry~(HMS), and study noncommutative projective spaces~\cite{HarKat}. Nadler has used them to prove HMS for certain Landau--Ginzburg models~\cite{Nadler}. They have been applied to study categorified Picard--Lefschetz theory by Katzarkov, Pandit, and~Spaide~\cite{KPS}. Bodzenta and Bondal have associated them to flops of curves~\cite{BB}. In previous work~\cite{Don}, I constructed them for wall crossings in geometric invariant theory (GIT).

Schobers are currently defined only on a few classes of spaces: \KandS{} define them on a disk, and indicate a generalization to Riemann surfaces. In this paper, I explain this generalization, and give simple operations on the resulting objects.

I give examples for wall crossings in GIT, and standard flops, demonstrating that functors which appear in these settings have natural interpretations in terms of schobers. My results are summarized in Subsection~\ref{sect results}.

In particular, I show that functors associated to a standard 3-fold flop, including the Bondal--Orlov flop functors, may be encoded in a schober on the projective line minus two points, thought of as a partially compactified stringy K\"ahler moduli space: I give a possible mirror symmetry interpretation of this in Subsection~\ref{sect MS}.

\begin{keyremark} Harder and Katzarkov~\cite{HarKat} give a related definition for schobers on Riemann surfaces. Continuing work of Dyckerhoff, Kapranov, Schechtman, and Soibelman seeks to further develop such definitions~\cite{DKSS}. Schobers on higher-dimensional spaces have been discussed only in restricted cases: see~\cite[Section~4]{KS2}, and upcoming work of Bondal--Kapranov--Schechtman~\cite{BKS}.
\end{keyremark}

\subsection{Results}\label{sect results}

One of the simplest classes of schobers consists of perverse sheaf of categories on a disk, singular at a point, known as `spherical pairs'. In previous work~\cite{Don}, I constructed spherical pairs for certain wall crossings in GIT stability. In this paper, under the same assumptions, I obtain a perverse sheaf on~$\mathbb{C}$, singular at~$i\mathbb{Z}$, in Theorem~\ref{keytheorem schober}.

\begin{keyremark} Perverse sheaves restrict to local systems on their smooth loci: analogously, schobers yield local systems of categories. Such local systems are indicated below by an action of a fundamental group on a category, or a fundamental groupoid on categories: full definitions are given in Section~\ref{subsection.loc_sys}.
\end{keyremark}

Take a toric GIT problem $V/G$ where $V$ is a smooth, equivariantly Calabi--Yau variety, and consider a `simple balanced' wall crossing. Informally, this is a wall crossing which affects only a single stratum of the GIT stratification: for details, see Definition~\ref{definition simple balanced}. Let $X_\pm$ denote the quotients on either side of the wall, and $\D(X)$ be the bounded derived category of coherent sheaves on~$X$.

\begin{keytheorem}\label{keytheorem schober}\emph{(Theorem~\ref{theorem schober}, Corollary~\ref{corollary schober})} For a simple balanced wall crossing as above, there exists a schober on $\mathbb{C}$, singular at $i\mathbb{Z}$, with
\begin{itemize}
\item generic fibre $\D(X _{-}  )\cong\D(X_{+} )$, and
\item an induced local system of categories on $\mathbb{C}-i\mathbb{Z}$ which may be presented in the following ways:
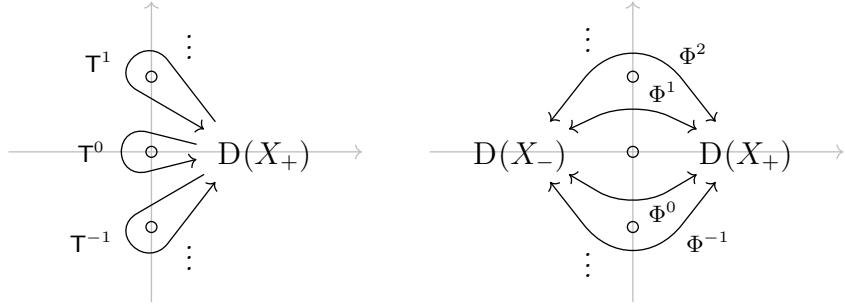
\begin{center}
\begin{tikzpicture}
 \node at (0,0) {$\monodromiesStandardFlop{$\D(X_+)$}{$\twistFun^{-1}$}{$\twistFun^{0}$}{$\twistFun^{1}$}$};
 \node at (6,0) {$\halfMonodromiesStandardFlop{$\D(X_+)$}{$\D(X_-)$}{\Phi}$};
 \end{tikzpicture}
\end{center}
\end{itemize}
The~$\twistFun^w$ are spherical twists obtained by Halpern-Leistner and Shipman~\cite{HLShipman}, and the~$\Phi^w$ are window equivalences as given by Halpern-Leistner~\cite{HL}, and Ballard--Favero--Katzarkov~\cite{BFK}.
\end{keytheorem}

The simplest cases of the theorem are toric standard flops, and certain flops of orbifold projective spaces: see Example~\ref{example vgit}.

\begin{keyremark} 
The base $\mathbb{C}$ of the schober above may be thought of as the complexification of the GIT stability space~$\mathbb{R}$ associated to the wall crossing. 
\end{keyremark}

\begin{keyremark} Theorem~\ref{keytheorem schober} holds under the weaker hypothesis that $V$ is smooth, and equivariantly Calabi--Yau, in a neighbourhood of a certain subvariety appearing in the GIT stratification: see Assumption~\ref{assumption.sph_pair} for details.
\end{keyremark}

I now give an application of Theorem~\ref{keytheorem schober} to standard flops. Take a standard flopping contraction $X \to Y$ with exceptional locus $E\cong \mathbb{P}^n$, so that the normal bundle of $E$ is the sum of $n+1$~copies of~$\cO_E(-1)$, and write $\D(X | E)$ for the full subcategory of $\D(X)$ whose objects have set-theoretic support on $E$.

\begin{keyproposition}\label{keytheorem.standard_flop}\emph{(Proposition~\ref{theorem.standard_flop})}
For a standard flopping contraction as above, the schober in Theorem~\ref{keytheorem schober} induces a schober on $\mathbb{C}$, singular at $i\mathbb{Z}$, with
\begin{itemize}
\item generic fibre $\D(X | E)$, and
\item monodromy around\, $iw \in i\mathbb{Z}$ given by the spherical twist of $\cO_E(w)$.
\end{itemize}
\end{keyproposition}

Restricting finally to standard $3$-fold flops, I construct and study a schober on a Riemann surface with non-trivial topology, namely \[M = \mathbb{P}^1 - \{\text{$3$ points}\}.\]
I understand this as a stringy K\"ahler moduli space~(SKMS), following a standard heuristic from mirror symmetry: see  Subsection~\ref{sect discussion} for discussion. Let~$\overline{M}$ be the partial compactification of $M$ with a point~$p$ replaced, so that \[\overline{M} = \mathbb{P}^1 - \{\text{$2$ points}\}.\] Writing $X'$ for the flop of~$X$, with exceptional locus $E'$, I prove the following.

\begin{keytheorem}\emph{(Theorem~\ref{theorem schober SKMS})}\label{keytheorem schober SKMS} For a standard $3$-fold flopping contraction as above, there exists a schober on $\overline{M}$, singular at~$p$, with
\begin{itemize}
\item generic fibres $\D(X | E)$ and $\D(X' | E')$, and
\item an induced local system of categories on $M$ given by
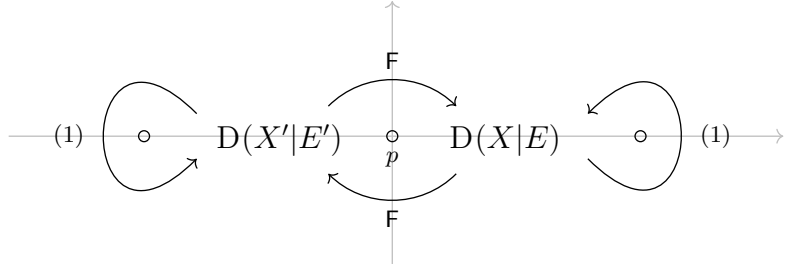
\begin{center}
\begin{tikzpicture}
 \node at (0,0) {$\monodromiesSKMS{$\Dres{X}{E}$}{$\Dres{X'}{E'}$}{$(1)$}{$(1)$}{$\flop$}{1}{1}{$p$}$};
\end{tikzpicture}
\end{center}
with Bondal--Orlov flop functors~ $\flop$, and line bundle twists~$(1)$.
\end{itemize}

\end{keytheorem}

I furthermore show that, in an appropriate sense, the pullback of the above local system on $M$ to a covering space isomorphic to $\mathbb{C}-i\mathbb{Z}$ is given by the schober from Proposition~\ref{keytheorem.standard_flop}: see Proposition~\ref{proposition schober SKMS}. I give a criterion to extend the  schober of Theorem~\ref{keytheorem schober SKMS} from $\overline{M}$ to $\mathbb{P}^1$ in a final Subsection~\ref{subsection compactif}.

\begin{keyremark} Work of Halpern-Leistner and Sam~\cite{HLSam} 
gives a local system of categories on an SKMS for an extensive class of examples, namely quotients for quasi-symmetric representations of reductive groups: it would be interesting to study schobers extending these.
\end{keyremark}

\subsection{Mirror symmetry of schobers}\label{sect MS} I outline a possible homological mirror symmetry~(HMS) application for Theorem~\ref{keytheorem schober SKMS}. Given a standard $3$-fold~flop between~$X$ and~$X'$, HMS may take the form of equivalences \[\D(X|E) \cong \F(Y) \qquad \text{and} \qquad \D(X'|E') \cong \F(Y'),\]  where $\F$ is an appropriate Fukaya category, and $Y$ and $Y'$ are mirror symplectic manifolds.

\begin{keyremark} Taking $X$ and $Y$ to be the resolved and deformed conifold respectively, with certain loci removed, an equivalence as above is proved by Chan, Pomerleano, and~Ueda~\cite{CPU}.
\end{keyremark}

Under mirror symmetry the SKMS~$M$, denoted by $\MB$ below, should be identified with a complex structure moduli space $\MA$, which is the base of a mirror family with $Y$ and $Y'$ as fibres. We may further identify  partial compactifications $\cMB \cong \cMA$. Now assume given an extension of the mirror family on  $\MA$ to $\cMA$, allowing singularities over the boundary, and denote this symplectic fibration by
\[f: \mathcal{Y} \to \cMA.\]
Letting  $\QB$ be the schober on $\cMB$ from Theorem~\ref{keytheorem schober SKMS}, we have the following.
\begin{question}Can a schober $\QA$ on $\cMA$ be constructed from the extended mirror family~$f$ such that there is an equivalence of schobers \[\QB \cong \QA\,?\]
\end{question}
The schober $\QA$ should have generic fibres $\F(Y)$ and $\F(Y')$ by construction, and the equivalence of schobers should respect the identification $\cMB \cong \cMA$, and the HMS equivalences above. I hope this question will be addressed in future work, along with more general examples.

\begin{keyremark}  The mirror operation to the standard $3$-fold flop given by Fan, Hong, Lau, and~Yau~\cite{FHLY} should induce half-monodromy equivalences for the schober~$\QA$ proposed above.\end{keyremark}

\subsection{Outline} In Section~\ref{section schober disk}, I review definitions of schobers on a disk, due to \KandS{}; Section~\ref{subsection.loc_sys} discusses local systems of categories and operations on them. Section~\ref{section schobers riemann} uses these notions to define schobers on Riemann surfaces, following ideas of~\KS{}, and gives criteria for extending these objects from an open subset. In~Section~\ref{section GIT}, I~review my previous work~\cite{Don} constructing spherical pairs from GIT wall crossings. Section~\ref{section examples plane} constructs schobers on $\mathbb{C}$, singular at~$i\mathbb{Z}$, proving Theorem~\ref{keytheorem schober} and Proposition~\ref{keytheorem.standard_flop}. Finally, in Section~\ref{section SKMS}, I prove results for the SKMS discussed above, in particular Theorem~\ref{keytheorem schober SKMS}.

\begin{acknowledgements}I am grateful to M.~Kapranov for inspiring conversations. I thank A.~Bondal, Y.~Ito, A.~King, S.~Meinhardt, E.~Segal, and M.~Wemyss for useful discussions, and J.~Stoppa and B.~Fantechi for their hospitality and interest in my work at SISSA, Trieste. I am grateful to an anonymous referee, and to P.~Schapira, for helpful comments. Finally, I thank the organizers of the 2016 Easter Island workshop on algebraic geometry for the opportunity to attend.\end{acknowledgements}

\begin{conventions}Varieties $X$ are assumed quasiprojective, with bounded derived category of coherent sheaves denoted by $\D(X)$. For a subvariety $Z$ of $X$, write $\Dres{X}{Z}$ for the full subcategory of $\D(X)$ consisting of objects with (set-theoretic) support on $Z$.
\end{conventions}

\section{Schobers on a disk}
\label{section schober disk}

In this section I recall categorifications, due to \KS{}, of the data of a perverse sheaf on a disk with singular points. For the case of a single singular point we have two categorifications, namely a `spherical pair' and a `spherical functor'. I explain how spherical pairs yield spherical functors. 

\subsection{Perverse sheaves}
\label{subsection.perv_sheaves}

The following is a standard description of the category $\operatorname{Per} (\Delta,B)$ of perverse sheaves on a disk $\Delta \subset \mathbb{C},$ possibly singular at points $B=\{b_1, \dots, b_n\}$. 

\begin{proposition}\label{proposition.GMV_perverse_description} \cite[Proposition~1.2]{GMV} The category $\operatorname{Per} (\Delta,B)$ is equivalent to the category of diagrams of vector spaces $D$, $D_1$, \dots, $D_n$ as follows, such that each endomorphism $m_i = \Idmatrix - v_i u_i $ of $D$ is an isomorphism.
\[
\begin{tikzpicture}
	\node (one) at (0,1) {$D_n$};
	\node (zero) at (0,0.1) {$\vdots$};
	\node (minusone) at (0,-1) {$D_1$};
	\node (plus) at (2,0) {$D$};
	\draw[->,transform canvas={xshift=+\arrowsplit,yshift=+\arrowsplit}] (plus) to  node[above] {$\scriptstyle u_n $} (one);
	\draw[<-,transform canvas={yshift=-\arrowsplit}] (plus) to  node[below] {$\scriptstyle v_n $} (one);
	\draw[->,transform canvas={xshift=-\arrowsplit,yshift=+\arrowsplit}] (plus) to  node[above] {$\scriptstyle u_1 $} (minusone);
	\draw[<-,transform canvas={yshift=-\arrowsplit}] (plus) to  node[below] {$\scriptstyle v_1 $} (minusone);
\end{tikzpicture}
\]
\end{proposition}

An equivalence may be constructed as follows. Fix~$x \in \partial\Delta,$ and choose a skeleton~$K$ as in Figure~\ref{figure abs per sheaf eg}. Then for $P$ in $\operatorname{Per} (\Delta,B)$ take the $D_i$ and $D$ to be the stalks at~$b_i$ and~$x$ respectively of the sheaf $\mathbb{H}^1_K(P)$ of cohomology with support on $K$. The $v_i$ are generalization maps, as described for instance in \cite[Section~1D]{KS1}, and the $u_i$ are dual to them. Other equivalences may be obtained by choosing a different skeleton, for instance~$K'$ in Figure~\ref{figure abs per sheaf eg}.

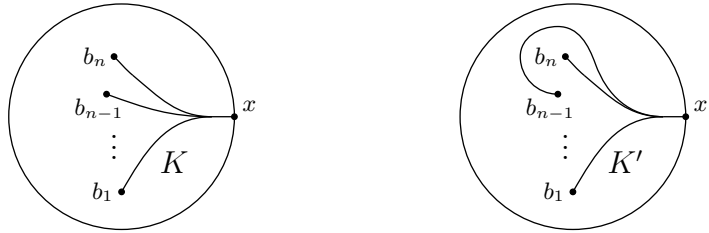
\begin{figure}[h]
\begin{center}
\begin{tikzpicture}
 \node at (0,0) {$\skeletonForDisk{b}{$x$}{$K$}$};
  \node at (6,0) {$\skeletonForDiskTwisted{b}{$x$}{$K'$}$};
\end{tikzpicture}
\end{center}
\caption{Skeleton examples.}\label{figure abs per sheaf eg}
\end{figure}

\begin{notation}\label{notation skeleta} Write $\mathcal{C}$ for the set of isotopy classes of skeleta for $(\Delta,B)$: here a~\emph{skeleton} is a union of simple arcs joining $x$ to the~$b_i$, coinciding near~$x$. \end{notation}

\begin{remark} The $n$-stranded braid group $\operatorname{Br}_n$ may be presented  as the isotopy classes of orientation-preserving diffeomorphisms of~$\Delta$ which fix $x$ and preserve $B$, giving a simply transitive action of~$\operatorname{Br}_n$ on~$\mathcal{C}$.

Under this description, we may choose generators $\sigma_i$ for $i=1 \dots n-1$ by choosing diffeomorphisms of~$\Delta$ which act on $B$ by swapping $b_i$ and $b_{i+1}$, and satisfy Artin braid relations
\begin{equation*}
\sigma_i \sigma_{i+1} \sigma_i  = \sigma_{i+1} \sigma_i \sigma_{i+1}.
\end{equation*}
For instance, in Figure~\ref{figure abs per sheaf eg}, for each $i$ take a disk $\delta$ with a diameter the line segment from~$b_i$ to~$b_{i+1}$, and a slightly larger disk $\delta_\epsilon$ with the same centre. Then a suitable diffeomorphism $\sigma_i$ may be constructed which acts on $\delta$ by a half turn in the counterclockwise direction, and as the identity on $\Delta - \delta_\epsilon$, so that, in particular, $\sigma_{n-1}$ sends $K$ to the isotopy class of $K'$.
\end{remark}

In the special case $B = \{ b \},$ we have the following alternative description due to \KS{}.

\begin{proposition}\label{proposition.KS_perverse_description} \cite[Section~9]{KS1} The category $\operatorname{Per} (\Delta,\{b\})$ is equivalent to the category of diagrams of vector spaces
\[
\begin{tikzpicture}
	\node (zero) at (0,0) {$E_0$};
	\node (plus) at (2,0) {$E_+$};
	\node (minus) at (-2,0) {$E_-$};
	\draw[->,transform canvas={yshift=+\arrowsplit}] (minus) to  node[above] {$\scriptstyle u_- $} (zero);
	\draw[<-,transform canvas={yshift=-\arrowsplit}] (minus) to  node[below] {$\scriptstyle v_- $} (zero);
	\draw[->,transform canvas={yshift=+\arrowsplit}] (plus) to  node[above] {$\scriptstyle u_+ $} (zero);
	\draw[<-,transform canvas={yshift=-\arrowsplit}] (plus) to  node[below] {$\scriptstyle v_+ $} (zero);
\end{tikzpicture}
\]
such that
\begin{enumerate}
\item\label{proposition.KS_perverse_description 1} $v_\pm  u_\pm = \Idmatrix$, and
\item\label{proposition.KS_perverse_description 2} $v_+  u_-$ and $v_-  u_+$ are isomorphisms.
\end{enumerate}

\end{proposition}

An equivalence is obtained as follows. Fix~$x_\pm \in \partial\Delta,$ and choose a skeleton~$L$ as in Figure~\ref{figure sph pair skeleton}. Then for $P$ in $\operatorname{Per} (\Delta,0)$ take $E_{\pm}$~and~$E_0$ to be the stalks at~$x_\pm$ and~$b$ respectively of the sheaf $\mathbb{H}^1_K(P)$. Again, the maps $v_\pm$ are generalization maps, and the $u_\pm$ dual to them.

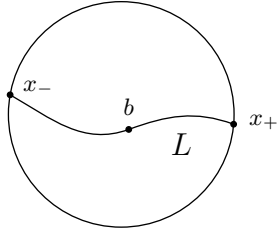
\begin{figure}[h]
\begin{center}
\begin{tikzpicture}
 \node at (0,0) {$\skeletonForPair{$x_-$}{$x_+$}{$L$}{$b$}$};
\end{tikzpicture}
\end{center}
\caption{Spherical pair skeleton.}\label{figure sph pair skeleton}
\end{figure}

\subsection{Schobers with multiple singular points}
\label{subsection schober multiple}

This subsection explains a categorification, due to~\KS{}, of the data of a perverse sheaf~$P$ on $\Delta$, possibly singular on some finite set of points $B$. I extend this to contractible domains in $\mathbb{C}$ with certain countable singular~sets in Subsection~\ref{subsection inf schober on C}.

The following gives a categorification of the vector space data associated to the perverse sheaf from Proposition~\ref{proposition.GMV_perverse_description}.

\begin{definition}\label{definition Kshadow}\cite[Section~2B]{KS2} A \defined{shadow} of a schober on~$(\Delta,B)$ is a collection of triangulated categories $\mathcal{D}$, $\mathcal{D}_1$, \dots, $\mathcal{D}_n$ with spherical functors \[\sphFun_i\colon \mathcal{D}_i \to \mathcal{D}.\]
\end{definition}

Each spherical functor $\sphFun$ induces a twist symmetry of $\mathcal{D}$ as follows
\begin{equation}\label{equation twist def}\twistFun_\sphFun := \operatorname{\sf Cone}\big(\,\sphFun \compose \sphFun^{\Radj} \xrightarrow{\text{\,{counit}\,\,}} \Id\big)\end{equation}
where $\sphFun^{\Radj}$ denotes a right adjoint. Here we assume given enhancements of our triangulated categories, so that functorial cones make sense, as in for instance~\cite[Appendix~A]{KS2}. For brevity, we have the following.

\begin{notation} Write $\twistFun_i$ for the twist symmetry $\twistFun_{\sphFun_i}$.\end{notation}

The vector space data from Proposition~\ref{proposition.GMV_perverse_description} depends on a choice of isotopy class of skeletons. Write~$K$ for such a class. To categorify a perverse sheaf we will therefore take, in Definition~\ref{schober}, a shadow $\cP_K$ for each class~$K$, satisfying appropriate compatibilities. For brevity, we refer to such a $\cP_K$ as a \defined{$K$-shadow}. The following Definition~\ref{transformation for Kshadows} gives operations relating $K$-shadows with $K'$-shadows for a different isotopy class $K' = \sigma K$, where $\sigma \in \operatorname{Br}_n$.

\begin{definition}\label{transformation for Kshadows}
For a $K$-shadow $\cP_K$, define a ${\sigma_i} K$-shadow $f_{\sigma_i} \cP_K$ by:
\begin{align*}
\sphFun'_{i}  & = \sphFun_{i+1}, \\
\sphFun'_{i+1}  & = \twistFun^{-1}_{i+1} \circ \sphFun_i, \\
\sphFun'_j & = \sphFun_j \text{\, for \,}j\neq i,i+1.
\end{align*}
For general $\sigma \in \operatorname{Br}_n$, define $f_{\sigma}$ as a composition of $f_{\sigma_i}$ and their inverses.
\end{definition}

 \begin{remark} The operations $f_{\sigma_i}$ are modelled on the transformation rules for vector space data from Proposition~\ref{proposition.GMV_perverse_description} under generators $\sigma_i$ of $\operatorname{Br}_n$, as given for instance in \cite[Proposition~1.3, and Remark below]{GMV}. In particular, the following braid relation may be easily checked for $i=1 \dots n-1$.
 \[ f_{\sigma_i} f_{\sigma_{i+1}} f_{\sigma_i} \cP_K \cong f_{\sigma_{i+1}} f_{\sigma_i} f_{\sigma_{i+1}} \cP_K \]
The key here is that, for a spherical functor $\sphFun\colon \mathcal{C} \to \mathcal{D}$ and a general equivalence $\Psi\colon \mathcal{D} \to \mathcal{E}$, we have
 \[ \twistFun_{\Psi \circ \sphFun} \cong \Psi \circ \twistFun_{\sphFun} \circ \Psi^{-1}.\]
\end{remark}

We then make the following definition.

\begin{definition}\label{schober}\cite[Section~2B]{KS2} A \defined{schober} $\cP$ on $(\Delta,B)$ is the data of a shadow~$\cP_K$ for each isotopy class~$K$ of skeletons such that there are identifications \[g_{K,\sigma} : f_\sigma \cP_K \overset{\sim}{\longrightarrow} \cP_{\sigma K}\]
for each $\sigma \in \operatorname{Br}_n$, compatible in the natural way.
\end{definition}

\begin{remark} The compatibility above is that, for $\tau, \sigma \in \operatorname{Br_n}$, composing 
\[f_\tau(g_{K,\sigma}) : f_\tau f_\sigma \cP_K
\longrightarrow f_\tau \cP_{\sigma K} \qquad\text{and}\qquad g_{\sigma K,\tau} : f_\tau \cP_{\sigma K}
\longrightarrow \cP_{\tau(\sigma K)}\] 
should coincide up to isomorphisms with
$ g_{K,\tau \sigma} : f_{\tau \sigma} \cP_K
\longrightarrow \cP_{(\tau\sigma)K}.$
\end{remark}

\begin{remark} Spherical functors are exactly schobers on $(\Delta,\{b\})$, as there is a unique isotopy class $K$ of skeletons in this case.\end{remark}

\subsection{Spherical pairs}\label{sect sph pair} This subsection presents a  categorification, again due to~\KS{}, of the data of a perverse sheaf~$P$ on~$\Delta$, possibly singular at a point $b$. This is a categorification of the vector space data of Proposition~\ref{proposition.KS_perverse_description}. It gives an alternate categorification for the case $B=\{b\}$ of the previous Subsection~\ref{subsection schober multiple}.

Recall that a semi-orthogonal decomposition
\begin{equation*}
\cD = \big\langle \cD_1, \cD_2 \big\rangle,
\end{equation*}
 is determined by embeddings $\delta_i \colon \cD_i \to \cD$ and induces projection functors
 \[\delta_1^{\Ladj}\colon \cD \longrightarrow \cD_1 \qquad\text{and}\qquad \delta_2^{\Radj}\colon \cD \longrightarrow \cD_2 \]
given by adjoints. We may then make the following definition.

\begin{definition}\label{definition sph pair}\cite{KS2}
A \defined{spherical pair} $\cP$ is  a triangulated category $\cP_0$ with admissible subcategories $\cP_\pm$ and semi-orthogonal decompositions
\begin{equation*}
\big\langle \cQ_-, \cP_- \big\rangle = \cP_0 = \big\langle \cQ_+, \cP_+ \big\rangle,
\end{equation*}
such that equivalences
 \[ \cQ_- \longleftrightarrow \cQ_+ \qquad\text{and}\qquad \cP_- \longleftrightarrow \cP_+, \]
are obtained by composition of the embeddings and projections above.
\end{definition}

Spherical pairs give spherical functors, as follows.

\begin{proposition}\label{proposition.spherical_pair_to_functor} \cite[Propositions~3.7, 3.8]{KS2} Given a spherical pair as above, the canonical functor
\[
\cP_- \longrightarrow \cQ_+
\]
is spherical. Its twist acts by the composition of canonical functors
 \[ \cQ_+ \longrightarrow \cQ_- \longrightarrow \cQ_+. \]
 \end{proposition}

\section{Local systems}
\label{subsection.loc_sys}

In this section I give definitions for local system of categories, and simple operations on them. These objects arise, in particular, by restricting a schober to its smooth locus, and will be used to describe and compare schobers later.

\subsection{Definitions}\label{subsection local sys def} We take the following.

\begin{definition} An $X$-coordinatized \defined{local system of categories} on a path-connected manifold $M$, for basepoints $X \subset M$, is given by
\begin{itemize}
\item a category $\mathcal{D}_i$ for each $x_i \in X$, and
\item an action of the fundamental groupoid $\pi_1(M,X)$ on $\{\mathcal{D}_i\}$.
\end{itemize}
\end{definition}

Concretely, we may take paths $ x_i \to x_j$ generating the groupoid, and corresponding functors $\mathsf{F}_{ij} \colon \mathcal{D}_i \to \mathcal{D}_j$ satisfying the groupoid relations, up to natural isomorphism: in particular, for a single basepoint $x$ we have simply an action of the fundamental group $\pi_1(M,x)$ on a category $\cD$.

We assume local systems as above to be \emph{triangulated} in the natural way.

\begin{definition}\label{definition iso} Two local systems are \emph{isomorphic} if their categories have equivalences $\mathcal{D}_i \cong \mathcal{D}'_i$, which intertwine their functors $\mathsf{F}_{ij}$ and $\mathsf{F}'_{ij}$. \end{definition}

\begin{definition}\label{definition strong iso} Two local systems are \emph{strongly isomorphic} if\, $\mathcal{D}_i = \mathcal{D}'_i$, and their $\mathsf{F}_{ij}$ and $\mathsf{F}'_{ij}$ are naturally isomorphic. \end{definition}

It would be desirable to allow isomorphisms in a suitable sense between local systems with different sets of basepoints: we do not treat this systematically here, however we supply the following notion for use later.

\begin{definition}\label{definition refinement} Take sets of basepoints $X \subset Y$ on $M$. We say that a \mbox{$Y$-coordinatized} local system is a \defined{refinement} of an $X$-coordinatized one if their actions are compatible under $\pi_1(M,X) \subset \pi_1(M,Y)$.\end{definition}

\begin{remark} A further notion of local system, not considered here, would relate natural isomorphisms between the $\mathsf{F}_{ij}$ to the path $2$-groupoid.\end{remark}

We have the usual notion of pullback of local systems, as follows.

\begin{definition}Take a continuous map $f\colon M\to N$ of manifolds, and a \mbox{$Y$-coordinatized} local system on $N$. An $X$-coordinatized local system on~$M$, for~$X=f^{-1}Y$, is induced by the map 
$ \pi_1(M,X) \to \pi_1(N,Y). $
\end{definition}
In particular, pullback along inclusions gives a notion of restriction.

\subsection{Local systems from schobers}

Perverse sheaves of vector spaces restrict to local systems away from their singularities. The following is an analog for schobers on a disk.

For some skeleton $K$, let $\gamma_i$ be a loop based at $x$, going counter-clockwise around $b_i$ in a small neighbourhood, but otherwise contained within $K$. Then $\pi_1(\Delta-B, x)$ is the free group with generators $\gamma_i$.

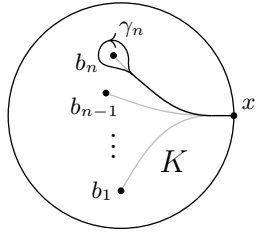
\smallskip

\begin{center}
\begin{tikzpicture}
 \node at (0,0) {$\loopForSkeleton{b}{$x$}{$K$}{$\gamma_n$}$};
\end{tikzpicture}
\end{center}

\begin{definition}\label{definition schober restriction} Given a schober on $(\Delta,B)$ as in Definition~\ref{schober}, define an $\{x\}$-coordinatized local system on $\Delta-B$ by, after choosing some $K$-shadow,
\begin{itemize}
\item assigning the category $\cD$ to $x$, and
\item letting $\gamma_i$ act by the twist $\twistFun_i$.
\end{itemize}
\end{definition}
\begin{remark} Up to isomorphism, this local system is independent of $K$ by construction of the braid group action on shadows from Definition~\ref{transformation for Kshadows}.\end{remark}

For a spherical pair we obtain a local system on $\Delta-\{b\}$ as follows. Fix two basepoints $x_\pm \in \partial\Delta$.

\begin{definition}\label{definition sph pair restriction} Given a spherical pair $\cP$ as in Definition~\ref{definition sph pair}, define an $\{x_\pm\}$-coordinatized local system on $\Delta-\{b\}$ by
\begin{itemize}
\item assigning the category $\cQ_\pm$ to $x_\pm$, and 
\item letting minimal counter-clockwise paths $x_\mp \to x_\pm$ act by equivalences
 \[ \cQ_\mp \longrightarrow \cQ_\pm \]
from Definition~\ref{definition sph pair}.
\end{itemize}
\end{definition}

\begin{remark} An $\{x_+\!\}$-coordinatized local system may be obtained from a spherical pair by taking  the spherical functor provided by Proposition~\ref{proposition.spherical_pair_to_functor}, interpreting it as a schober on $(\Delta,\{b\})$, and using Definition~\ref{definition schober restriction}. The local system from Definition~\ref{definition sph pair restriction} above is a refinement of this, by the latter part Proposition~\ref{proposition.spherical_pair_to_functor}.\end{remark}

\begin{remark} Clearly a dual local system may be obtained in Definition~\ref{definition sph pair restriction} by taking $\cP_\pm$ instead of $\cQ_\pm$, however we do not use this here.\end{remark}

\section{Riemann surfaces}
\label{section schobers riemann}

In this section, I give a definition of a schober on a Riemann surface, developing details of the definition indicated by \KS{} in \cite[Section~2E]{KS2}. A related definition, using more general skeleta than those considered here, is given by Harder and Katzarkov in \cite[Section~3.4]{HarKat}.

\subsection{Definitions} Let $\Sigma$ be a Riemann surface, possibly with boundary. Take a finite subset~$B$ of its interior, and choose furthermore a disk~$\Delta$ in its interior which contains~$B$. We will define a schober on $\Sigma$, possibly singular on $B$, as the data of a schober on $\Delta$ and a local system on $\overline{\Sigma - \Delta}$, agreeing on their intersection.

\begin{remark} Though the definitions in this section suffice for our present purposes, they are somewhat unwieldy, and should be considered as preliminary: improved notions would remove the choice of disk~$\Delta$, as in \cite[Section~3.4]{HarKat}, and could incorporate a categorification of `parafermionic' descriptions of perverse sheaves by \KS{}~\cite{KS3}. This is currently being pursued by Dyckerhoff, Kapranov, Schechtman, and Soibelman~\cite{DKSS}.
\end{remark}

Fix as before a basepoint $x$ in $\partial \Delta$, allow a finite (possibly empty) set of further basepoints $Y$ in $\Sigma - \Delta,$  and let $X=Y \cup \{ x \}.$ 

\begin{definition}\label{definition schober surface}\cite[Section~2E]{KS2} A \defined{schober} on $(\Sigma,B)$ is the data:
\begin{enumerate} 
\item\label{definition schober surface 1} a schober on $(\Delta,B)$ as in Definition~\ref{schober};

\item\label{definition schober surface 2} an $X$-coordinatized local system of categories on $\overline{\Sigma - \Delta}$; 

\item\label{definition schober surface 3} a strong isomorphism of the induced $\{ x \}$-coordinatized local systems on~$\partial \Delta$.
\end{enumerate}
\end{definition}

\begin{remark} The notion of strong isomorphism is from Definition~\ref{definition strong iso}: in particular, the local systems assign isomorphic categories to $x$.\end{remark}

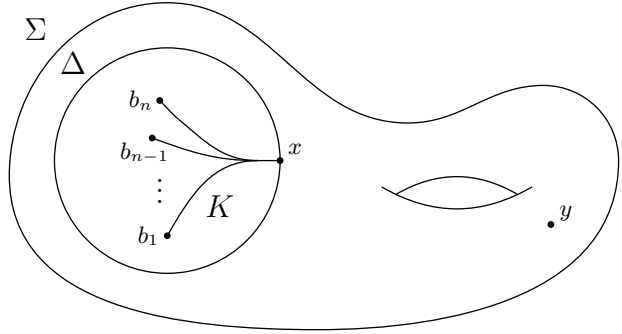
\begin{figure}[h]
\begin{center}
\begin{tikzpicture}
 \node at (0,0) {$\skeletonForSurface{b}{$x$}{$K$}{$y$}{$\Delta$}{$\Sigma$}$};
\end{tikzpicture}
\end{center}
\caption{Riemann surface skeleton example.}
\end{figure}

\begin{remark} In Definition~\ref{definition schober surface}, after a choice of $K$-shadow, the induced local system on $\partial \Delta$ is determined by the action on $\cD$ of \begin{equation}\label{equation compos}\twistFun_1 \compose \dots \compose \twistFun_n.\end{equation}\end{remark}

\begin{example} Let $\Sigma=\mathbb{P}^1$. Then the subset $\overline{\Sigma - \Delta}$ is contractible, and so the `monodromy at infinity' (\ref{equation compos}) for a schober on $\Sigma$ is isomorphic to the identity: we give an example in Subsection~\ref{subsection compactif}.\end{example}

Restricting to a single marked point $b$, and fixing basepoints $x_\pm$ in $\partial \Delta$, we have the following.

\begin{definition}\label{definition schober pair surface} A \defined{schober of spherical pair type} on $(\Sigma,\{b\})$ is the data: 
\begin{enumerate} 
\item\label{definition schober pair surface 1} a spherical pair as in Definition~\ref{definition sph pair};
\item\label{definition schober pair surface 2} a local system as in (\ref{definition schober surface 2}) above, but with $X=Y \cup \{ x_\pm \}$;
\item\label{definition schober pair surface 3} a strong isomorphism as in (\ref{definition schober surface 3}) above, but of $\{ x_\pm \}$-coordinatized local systems.
\end{enumerate}
\end{definition}

By construction, we have the following.

\begin{proposition}A schober of spherical pair type on $(\Sigma,\{b\})$ induces a schober on $(\Sigma,\{b\})$.
\begin{proof} We give the required data for Definition~\ref{definition schober surface}. The schober on $(\Delta,\{b\})$ for Definition~\ref{definition schober surface}(\ref{definition schober surface 1}) can be taken as the spherical functor associated to the spherical pair  provided by Definition~\ref{definition schober pair surface}(\ref{definition schober pair surface 1}), using Proposition~\ref{proposition.spherical_pair_to_functor}. A~suitable local system for Definition~\ref{definition schober surface}(\ref{definition schober surface 2}), which is $Y \cup \{ x \}$-coordinatized, may be induced from the $Y \cup \{ x_\pm \}$-coordinatized local system provided by \mbox{Definition~\ref{definition schober pair surface}(\ref{definition schober pair surface 2}).}
\end{proof}
\end{proposition}

We also have the following restriction result.

\begin{proposition} A schober on $(\Sigma,B)$ induces an $X$-coordinatized local system of categories on $\Sigma - B$.
\begin{proof}Definition~\ref{definition schober restriction} gives an $\{x\}$-coordinatized local system on  $\Delta-B$, and  Definition~\ref{definition schober surface}(\ref{definition schober surface 2}) gives an $X$-coordinatized local system on $\overline{\Sigma - \Delta}$, which glues to it along $\partial \Delta$ by Definition~\ref{definition schober surface}(\ref{definition schober surface 3}).
\end{proof}
\end{proposition}

The same result holds for schobers of spherical pair type.

\subsection{Extension results}\label{Sect ext res} We give simple criteria to extend schobers over points on Riemann surfaces.

Consider an $X$-coordinatized local system on $\Sigma^p$, where $\Sigma^p = \Sigma - \{p\}$, for $p$ in the interior of~$\Sigma$. Let $\Delta$ be a disk neighbourhood of $p$ also contained in the interior of~$\Sigma$, where the only points of $X$ in $\Delta$ are two points $x_\pm \in \partial \Delta$. The following is a direct consequence of the definitions.

\begin{proposition}\label{proposition schober pair extend} Take an 
$X$-coordinatized local system on $\Sigma^p$ as above. Given a spherical pair inducing the same local system on $\partial \Delta$, we obtain a schober of spherical pair type on $(\Sigma, \{p\})$
\end{proposition}

Now consider a schober on $(\Sigma^p ,B)$. Let $C$ be a simple arc joining $x$ and~$p$ in $\overline{\Sigma - \Delta}$, not meeting $X$. Let $\gamma_0$ be a loop based at $x$, going counter-clockwise around $p$ in a small neighbourhood $N$, but otherwise contained in~$C$, and let $\mathsf{F}$ be the autoequivalence of $\cD$ associated to $\gamma_0$.

\begin{proposition}\label{proposition schober extend} Take a schober on $(\Sigma^p ,B)$ as above.
\begin{enumerate}
\item\label{proposition schober extend 1}  Given a presentation of\, $\mathsf{F}$ as the twist of a spherical functor, we obtain a schober on $(\Sigma ,B_p)$, where $B_p = B \cup \{p\}$.
\item\label{proposition schober extend 2}  The two schobers induce isomorphic local systems on $\Sigma^p - B$.
\end{enumerate}

\begin{proof} Recall that the data of a perverse schober on $(\Sigma^p,B)$ includes the data of a perverse schober on $(\Delta,B)$. Choose a mutual neighbourhood $\Delta_p$ in $\Sigma$ of $\Delta$, $C$, and $N$. Require furthermore that this does not meet $Y$, and deformation retracts to $\Delta$ (we use that $C$ is a simple arc here). Then for each skeleton $K$ on $\Delta$ we obtain a skeleton $K_p$ on $\Delta_p$, as sketched below.

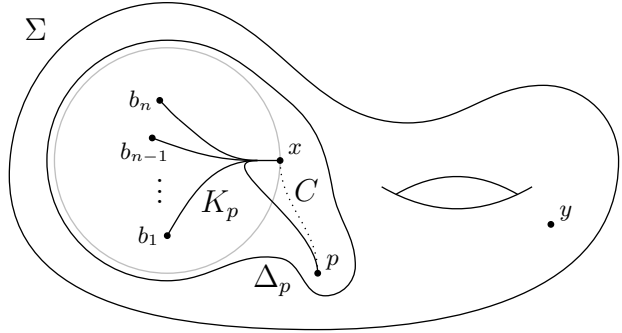
\begin{center}
\begin{tikzpicture}
 \node at (0,0) {$\skeletonForPuncturedSurface{b}{$x$}{$K_p$}{$y$}{$\Delta$}{$\Sigma$}{$p$}{$C$}{$\Delta_p$}$};
\end{tikzpicture}
\end{center}

From a $K$-shadow for $(\Delta ,B)$ given by spherical functors $\sphFun_i$, we obtain a $K_p$\,-shadow for $(\Delta_p,B_p)$ by appending the spherical functor $\sphFun$ supplied by assumption~(\ref{proposition schober extend 1}). A general shadow for $(\Delta_p,B_p)$ is obtained from such a~\mbox{$K_p$\,-shadow} by the action of the braid group $\operatorname{Br}_{n+1}$. To obtain a schober on $(\Delta_p,B_p)$, we thence obtain shadows using the transformation rule in Definition~\ref{transformation for Kshadows}, and verify the compatibilities of Definition~\ref{schober}.

The data of a perverse schober on $(\Sigma^p,B)$ includes the data of a local system on $\overline{\Sigma^p - \Delta}$, and we obtain a local system on $\overline{\Sigma - \Delta_p}$ by restriction. Choosing a skeleton $K_p$ as above for our schober on $(\Delta_p,B_p)$, the action of the induced local system on $\partial \Delta_p$ is determined by  \[\twistFun_\sphFun \compose \twistFun_1 \compose \dots \compose \twistFun_n. \] The definition of a schober on $(\Sigma ,B_p)$ requires that the same local system on $\partial \Delta_p$ is induced by the above local system on $\overline{\Sigma - \Delta_p}$, but this follows by assumption. This shows~(\ref{proposition schober extend 1}).

The claim~(\ref{proposition schober extend 2}), that the schobers induce isomorphic local systems on $\Sigma^p - B = \Sigma - B_p$, can be checked on the cover given by $\Delta - B$ and $\overline{\Sigma^p - \Delta}$.
\end{proof}
\end{proposition}

\begin{remark} We may take $B$ empty: then a schober on $(\Sigma^p, B)$ is just a local system on $\Sigma^p$, and the above Proposition~\ref{proposition schober extend} says that this extends to a schober on $(\Sigma, \{p\})$ if its monodromy around $p$ is a twist.
\end{remark}

\begin{remark} E.~Segal explains in \cite{Seg2} how all autoequivalences may be presented as a twist, by a dg-categorical construction. However, the criterion above remains relevant if we want to work in, for instance, the category of varieties and Fourier--Mukai kernels between them.\end{remark}

\subsection{Infinite singular points}
\label{subsection inf schober on C}
We consider a contractible domain $\mathbb{D}$ in $\mathbb{C}$, and define a perverse schober on $(\mathbb{D},B)$, with $B$ a countably infinite subset without accumulation points. We assume furthermore that $B$ is the image of an embedding $\mathbb{Z} \hookrightarrow \mathbb{D}$ which extends to an embedding of a strip~$\mathbb{R} \times I \hookrightarrow \mathbb{D}$, as explained below.

For the set $\mathcal{C}$ of isotopy classes of skeleta as in Notation~\ref{notation skeleta}, we restrict to \emph{finitary skeleta}, namely those obtained from a reference skeleton $K$ by an action of the infinite braid group with generators $\sigma_i$ for $i\in \mathbb{Z}$,
\[\operatorname{Br}_\infty = \big\langle\, \sigma_i \,\big|\, \sigma_i \sigma_{i+1} \sigma_i  = \sigma_{i+1} \sigma_i \sigma_{i+1}, \, \sigma_j \sigma_k = \sigma_k \sigma_j \text{ for $|j-k|>1$} \,\big\rangle. \]
To obtain such an action, proceed as follows. Take $I = [-1,+1] \subset\mathbb{R}$, and identify $\mathbb{Z}$ with the points $(i,0)$ of  $\mathbb{R} \times I$. We assume that the embedding $\mathbb{Z} \hookrightarrow \mathbb{D}$ extends to an embedding~$\mathbb{R} \times I \hookrightarrow \mathbb{D}$ under this identification.

Now for each $i$ take a disk~$\delta$ in the strip $\mathbb{R} \times I$ with a diameter the line segment from $(i,0)$ to $(i+1,0)$, and a slightly larger disk $\delta_\epsilon$ with the same centre. Then a diffeomorphism $\sigma_i$ of $\mathbb{R} \times I$ may be constructed which acts on $\delta$ by a half turn in the counterclockwise direction, and as the identity on the complement of $\delta_\epsilon$, so that the $\sigma_i$ satisfy the relations of $\operatorname{Br}_\infty$. These $\sigma_i$ may then be extended to diffeomorphisms $\sigma_i$ of $\mathbb{D}$, acting as the identity on the complement of $\mathbb{R} \times I$, via the assumed embedding~$\mathbb{R} \times I \hookrightarrow \mathbb{D}$, giving an action of action of $\operatorname{Br}_\infty$ on $\mathbb{D}$.

The induced action of $\operatorname{Br}_\infty$ on $\mathcal{C}$ is simply transitive by construction. The definition then proceeds as in Subsection~\ref{subsection schober multiple}.

\begin{remark} Given a schober on $(\mathbb{D},B)$, a local system on $\mathbb{D}-B$ may be obtained as in Definition~\ref{definition schober restriction}. The description used there of the fundamental group of $\mathbb{D}-B$ still holds, using the accumulation point assumption above. \end{remark}

\begin{remark}It would be more direct to consider all skeleta, not just finitary skeleta as above. Indeed for $(\mathbb{D},B)=(\mathbb{C},\mathbb{Z})$, it follows from work of Fabel~\cite{Fabel} that an appropriate mapping class group is a completion of~$\operatorname{Br}_\infty$, but we wish to avoid considering limits of functor compositions.\end{remark}

\section{Geometric invariant theory}
\label{section GIT}

I review results from GIT sufficient to recall, from previous work~\cite{Don}, a construction of a spherical pair for certain GIT wall crossings: this is Theorem~\ref{theorem.sph_pair}. For further details on GIT, I refer the reader to treatments of Halpern-Leistner~\cite{HL}, and Ballard, Favero, and Katzarkov~\cite{BFK}.

\subsection{Setting} We consider the following.

\begin{setup}\label{GIT setup} Take a projective-over-affine variety $V$ with
\begin{itemize}
\item an action of a torus~$G$, and
\item a linearization~$\cM$.
\end{itemize}
\end{setup}

This data determines a semistable locus $V^{\mathrm{ss}}(\cM) \subseteq V$: we take the \defined{\mbox{GIT quotient}} to be the quotient stack $V^{\mathrm{ss}}/G$. The complement of $V^{\mathrm{ss}}$ admits a GIT stratification by strata $S^i$, with associated one-parameter subgroups~$\lambda^i$ of~$G$: each~$S^i$ contains an open subvariety $Z^i$ of the $\lambda^i$-fixed locus.

Take a wall crossing between linearizations $\linzn_\pm$, with  linearization $\linzn_0$ on the wall. Then the semistable loci for $\linzn_\pm$ may be obtained from the locus for $\linzn_0$ by removing appropriate strata, say $S^i_\pm$. We use the following notion, specializing a situation considered by Halpern-Leistner~\cite[Definition 4.4]{HL}.

\begin{definition}\label{definition simple balanced} \cite[Definition~3.18]{Don} A wall crossing is \defined{simple balanced} if 
\[
V^{\mathrm{ss}}(\linzn_0) = V^{\mathrm{ss}}(\linzn_\pm) \cup S_\pm
\]
where $S_\pm$ are strata for the linearizations $\cM_\pm$, associated with equal fixed subvarieties $Z$, and inverse one-parameter subgroups, as follows. \[\lambda_- = (\lambda_+)^{-1}\]
\end{definition}
We adopt a convention that the subgroup $\lambda_-$ is associated to the wall, and write $\lambda=\lambda_-$ for brevity.

\smallskip

Consider a single stratum $S$, dropping indices~$i$ for simplicity. By construction, $S$ consists of the points of $X$ which flow to $Z$ under $\lambda$, inducing a $G$-equivariant morphism $\pi\colon S\to Z$.

\subsection{Derived categories} Assume for simplicity in this subsection that the GIT stratification consists of a single stratum.

As notation, for $F^{\bullet} \in \D(V/G)$, write $\operatorname{wt}_\lambda F^{\bullet}$ for the set of $\lambda$-weights which appear in some derived restriction of $F^{\bullet}$ to~$Z$. Then for an integer~$w$, a \emph{window} $\mathcal{G}^{[w,w+\eta)}$ is defined as the full subcategory of $\D(V/G)$ with objects
\begin{align*}
\mathcal{G}^{[w,w+\eta)} &= \big\{\, F^{\bullet} \in \D(V/G) \mathrel{\big|} \operatorname{wt}_\lambda F^{\bullet} \subseteq [w,w+\eta) \,\big\},
\end{align*}
where $\eta$ is given as follows.
\begin{definition}\label{definition window width} The \defined{window width} $\eta$ is the $\lambda$-weight of $\operatorname{det} \mathcal{N}^{\vee}_S V$ on $Z$.\end{definition}

Then we have the following theorem of Halpern-Leistner, and Ballard, Favero, and Katzarkov.

\begin{theorem}\label{theorem.splittings_of_res}\cite{HL,BFK}
The restriction functor to the GIT~quotient
\[ \operatorname{res} \colon \mathcal{G}^{[w,w+\eta)} \subset \D(V/G) \longrightarrow \D(V^{\mathrm{ss}}/G) \]
 is an equivalence.
\end{theorem}

Given a wall crossing, we may then make the following definition. Denote the GIT quotients corresponding to either side of the wall by $X_{\pm}.$
\begin{definition}\label{definition window equivs}Let the \defined{window equivalence} \[\Phi^w\colon \D(X_{\pm}) \longrightarrow \D(X_{\mp})\] be given by composition of equivalences from Theorem~\ref{theorem.splittings_of_res} as follows.
\[\D(X_{-}) \overset{\sim}{\longleftarrow} \mathcal{G}^{[w,w+\eta)} \overset{\sim}{\longrightarrow} \D(X_{+})\]
\end{definition}

The following lemma is given by Halpern-Leistner and Shipman. Write $\D(Z/G)^w$ for the full subcategory of $\D(Z/G)$ whose objects have $\lambda$-weight~$w$, and define a subcategory $\mathcal{G}^{[w,w+\eta]}$ of $\D(V/G)$ similarly to $\mathcal{G}^{[w,w+\eta)}$ above.

\begin{lemma}\label{lemma.iota_adjoints}\cite{HLShipman} Recall that we have $\pi\colon S \to Z$, and let $j$ be the inclusion of $S$ into $V$\!. For each $\lambda$-weight $w$, the following functor is an embedding.
\[  \iota = j_* \pi^* \colon \D(Z/G)^w \longrightarrow \mathcal{G}^{[w,w+\eta]} \]

\end{lemma}

\subsection{Spherical pair from wall crossing}

Assume given a simple balanced wall crossing for a toric GIT problem $V/G$ as in Setup~\ref{GIT setup}. In this subsection I review a construction of a spherical pair in this setting, dependent on an integer weight $w$.

Recall that the wall crossing has an associated subvariety~$Z$, and  one-parameter subgroup~$\lambda$.

\begin{assumption}\label{assumption.sph_pair}\cite[Theorem~A]{Don} Assume that:

\begin{enumerate}
\item\label{theorem assumption 1} the variety $V$ is smooth in a $G$-equivariant neighbourhood of $Z$;
\item\label{theorem assumption 2} the canonical sheaf\, $\omega_V$ has $\lambda$-weight zero on $Z$.
\end{enumerate}
\end{assumption}

\begin{remark}
Note that this is satisfied in particular if $V$ is smooth and $G$-equivariantly Calabi--Yau.
\end{remark}

Recall that we take a wall crossing between linearizations~$\linzn_\pm$ with GIT quotients~$X_{\pm}$, and \mbox{linearization~$\linzn_0$} on the wall.  Window widths $\eta_\pm$ are defined to be the $\lambda_\pm$-weight of $\det \cN^\vee_{S_\pm} V$~on~$Z$. We then have the following.

\begin{theorem}\label{theorem.sph_pair}\cite[Theorem~A]{Don} Take a simple balanced wall crossing for a toric GIT problem $V/G$ satisfying Assumption~\ref{assumption.sph_pair}. It follows that $\eta_+=\eta_-$, and we write $\eta$ for their common value.

\renewcommand{\theenumi}{\arabic{enumi}}
\begin{enumerate}
\item\label{theorem.sph_pair 1} For each integer $w$, there exists a spherical pair $\cP$ determined by 
\[ \big\langle \D(X_{-})\,,\, \D(Z/G)^{\vphantom{\eta} w} \,\big\rangle = \mathcal{P}_0 = \big\langle \D(X_{+})\,,\, \D(Z/G)^{w+\eta} \,\big\rangle\]
where superscripts denote $\lambda$-weight subcategories, and $\mathcal{P}_0$ is a full subcategory with objects
\[
\mathcal{P}_0 = \left\{ \begin{array}{c|c} F^{\bullet} \in \D\!\big(V^{\mathrm{ss}}(\linzn_0)/G\,\big) & \operatorname{wt}_{\lambda} F^{\bullet} \subseteq [w,w+\eta] \end{array} \right\}.
\]

\item\label{theorem.sph_pair 2} The spherical pair $\mathcal{P}$ induces window equivalences
\begin{align*}
\Phi^w & \colon \D(X_{-}) \longrightarrow \D(X_{+}) \\
\Phi^{w+1} & \colon \D(X_{+}) \longrightarrow \D(X_{-})
\end{align*}
obtained by applying Definition~\ref{definition window equivs} to $V^{\mathrm{ss}}(\linzn_0)/G$.
\end{enumerate}

\end{theorem}

\subsection{Spherical functors}\label{section sph functor}

We will use the following sequence of spherical functors obtained by Halpern-Leistner and Shipman~\cite[Section~3.2]{HLShipman}, and shown to arise from the spherical pair of Theorem~\ref{theorem.sph_pair} by the author~\cite{Don}.

\begin{corollary}\label{corollary sph functor}\cite[Corollary~4.6]{Don} In the setting of Theorem~\ref{theorem.sph_pair}, the given spherical pair induces a spherical functor 
\[
\sphFun^w =  \operatorname{res}_+ \circ \, \iota_{-} \colon \D(Z/G)^w  \longrightarrow \D(X_{+}),
\]
where $\iota_{-}$ is obtained from Lemma~\ref{lemma.iota_adjoints}, applied to~$V^{\mathrm{ss}}(\linzn_0)/G$.\end{corollary}

\begin{notation}\label{notation sph functor} Write $\twistFun^w$ for the twist of $\sphFun^w$.\end{notation}

\begin{proposition}\label{proposition derived monodromy}\cite[Proposition~3.7(1)]{Don} In the setting of Theorem~\ref{theorem.sph_pair}, we have an isomorphism
\[
\twistFun^w \cong \Phi^w \circ \Phi^{w+1}
\]
of autoequivalences of $\D(X_{+})$.
\end{proposition}

\section{Examples: complex plane}
\label{section examples plane}

In this section we obtain  schobers on $(\mathbb{C},i\mathbb{Z})$ associated with GIT wall crossings and standard flops.

\subsection{Schober from wall crossing}\label{subsection construction}

The spherical pair of Theorem~\ref{theorem.sph_pair} was dependent on a choice of integer weight~$w$: the result below should be seen as gluing  the resulting spherical pairs for different $w$ to obtain a perverse schober on $(\mathbb{C},i\mathbb{Z})$.

\begin{theorem}\label{theorem schober} Take a simple balanced wall crossing for a toric GIT problem $V/G$ satisfying Assumption~\ref{assumption.sph_pair}, with GIT quotients~$X_{\pm}$. There exists a schober~$\cP$  on $(\mathbb{C},i\mathbb{Z})$ with

\begin{itemize}
\item\label{theorem schober 1} generic fibre $\D(X _{+} )$, and 
\item\label{theorem schober 2} monodromy around\, $iw \in i\mathbb{Z}$ given by $\twistFun^w$ from Notation~\ref{notation sph functor}. \end{itemize}

\begin{proof}Taking a skeleton $K^0$ below, construct a $K^0$-shadow $\cP_{K^0}$\ by associating to $iw$, for $w\in\mathbb{Z}$, the following spherical functor from Corollary~\ref{corollary sph functor}.
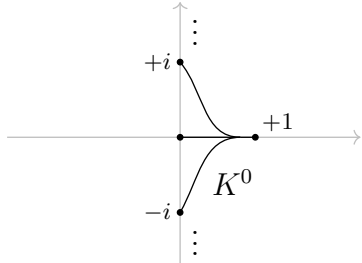
\[\sphFun^{w} \colon \D(Z/G)^w  \longrightarrow \D(X_{+}) \]

\begin{center}
\begin{tikzpicture}
 \node at (0,0) {$\skeletonForGIT{$K^0$}{$+1$}{$+i$}{$-i$}$};
\end{tikzpicture}
\end{center}

The braid group $\operatorname{Br}_\infty$ acts simply transitively on  the set of isotopy classes of finitary skeleta, so any such isotopy class $K$ is of the form $\sigma K^0$ for a unique~$\sigma\in\operatorname{Br}_\infty$. Then we may define a $K$-shadow (unique up to isomorphism) by applying the transformation rule for shadows from Definition~\ref{transformation for Kshadows} to~$\cP_{K^0}$, and verify compatibilities as in Definition~\ref{schober}.
\end{proof}
\end{theorem}

The following is then immediate, using Proposition~\ref{proposition derived monodromy}.

\begin{corollary}\label{corollary schober}In the setting of Theorem~\ref{theorem schober}, the restriction of $\cP$  to $\mathbb{C}-i\mathbb{Z}$ is a local system  of categories as follows, with a refinement as indicated.
\end{corollary}

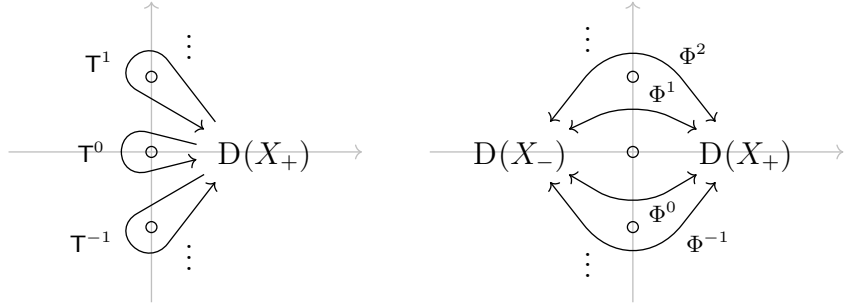
\begin{center}
\begin{tikzpicture}
 \node at (0,0) {$\monodromiesStandardFlop{$\D(X_+)$}{$\twistFun^{-1}$}{$\twistFun^{0}$}{$\twistFun^{1}$}$};
 \node at (6,0) {$\halfMonodromiesStandardFlop{$\D(X_+)$}{$\D(X_-)$}{\Phi}$};
\end{tikzpicture}
\end{center}

\begin{example}\label{example vgit} Take vector spaces $U_\pm$ with $\mathbb{C}^*$-actions having strictly positive weights with the same sum. Then $V=U_+ \oplus U_-^{\vee}$ is $\mathbb{C}^*$-equivariantly Calabi--Yau, and so the above applies. The quotients~$X_\pm$ are related by a flop exchanging orbifold weighted projective spaces~$\mathbb{P}U_+$ and~$\mathbb{P}U_-^{\vee}$.
\end{example}

\subsection{Schober for standard flop}

For a contraction $X\to Y$ let $E$ denote the exceptional locus, and recall that $\Dres{X}{E}$ is the full subcategory of $\D(X)$ whose objects have (set-theoretic) support on $E$.

\begin{proposition}\label{theorem.standard_flop} Take a standard flopping contraction $X\to Y$, so that in particular $E\cong \mathbb{P}^n$. Then the schober of Theorem~\ref{theorem schober} induces a schober $\cP$ on $(\mathbb{C},i\mathbb{Z})$ with 

\begin{itemize}
\item\label{theorem.standard_flop 1} generic fibre $\Dres{X}{E}$, and 
\item\label{theorem.standard_flop 2} monodromy around\, $iw \in i\mathbb{Z}$ given by the spherical twist $\twistFun^w$ of $\cO_E(w)$. 
\end{itemize}
\end{proposition}
\begin{proof}
We first prove the result for a local model of the flopping contraction, denoted $X_\pm \to Y_0$, as follows. Take a vector space $U$ with projectivization~$E$, and let $V=U\oplus U^\vee$ with $\mathbb{C}^*$-action induced by the scalar action on $U$. The two GIT quotients $X_\pm$ are the two sides of the flop, and have flopping contractions to the affine quotient $Y_0 = \operatorname{Spec} \mathbb{C}[V]^{\mathbb{C}^*}$. The exceptional locus in $X_+$ is $E$, given by the quotient of the subspace $U\!\oplus 0$ of $V$, after removing the unstable point~$0$.

Now $V$ is smooth and Calabi--Yau, so Theorem~\ref{theorem schober} gives a perverse schober $\cP$ on $(\mathbb{C},i\mathbb{Z})$ with generic fibre $\D(X_+)$. The fixed locus $Z$ is $0$, so the categories $\D(Z/G)^w$ appearing in the proof are just $\D(\operatorname{pt})$. Using Corollary~\ref{corollary sph functor}, we find that the monodromy around $iw$ is the spherical twist by $\cO_E(w)$. This lies in $\Dres{X_+}{E}$, and therefore the monodromies preserve the latter category: we thus obtain a schober on $(\mathbb{C},i\mathbb{Z})$ with generic fibre $\Dres{X_+}{E}$, and with the required monodromy property.

It remains to construct a similar schober with generic fibre $\Dres{X}{E}$. As $X\to Y$ and $X_+\to Y_0$ are standard flopping contractions with shared exceptional locus $E$, the formal completions of $\widehat{X}$ and $\widehat{X}_+$ along $E$ are isomorphic, and so $\D(\widehat{X})\cong\D(\widehat{X}_+)$. The images of the restriction functors 
\[\Dres{X}{E} \longrightarrow \D(\widehat{X}) \qquad\text{ and }\qquad \Dres{X_+}{E} \longrightarrow \D(\widehat{X}_+) \]
coincide under this equivalence. We conclude $\Dres{X}{E} \cong\Dres{X_+}{E}$, where this last equivalence identifies objects $\cO_E(w)$, and the result follows.\end{proof}

The restriction of the schober $\cP$  to $\mathbb{C}-i\mathbb{Z}$ is a local system as follows.

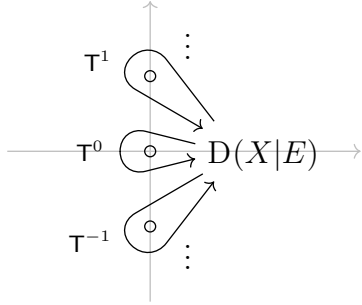
\begin{center}
\begin{tikzpicture}
 \node at (0,0) {$\monodromiesStandardFlop{$\Dres{X}{E}$}{$\twistFun^{-1}$}{$\twistFun^{0}$}{$\twistFun^{1}$}$};
\end{tikzpicture}
\end{center}

\section{Examples: K\"ahler moduli space}
\label{section SKMS}

In this section we study $X\to Y$ a standard $3$-fold flopping contraction, that is a contraction of a single $(-1,-1)$-curve $E$. There is currently no accepted mathematical definition of the stringy K\"ahler moduli space (SKMS) in general. However, as a heuristic for this example, we take 
\[M = \mathbb{P}^1 - \{\text{$3$ points}\}.\]
This heuristic is discussed below: we then construct schobers on partial compactifications of $M$, and relate them to the previous section.

\subsection{Discussion}\label{sect discussion} The space~$M$ may be obtained, following Toda, as a quotient of a certain space of normalized Bridgeland central charges on $\Dres{X}{E}$, see~\cite[end of Section~5.2, Example]{TodaRes}.  This turns out to be isomorphic to~$\mathbb{C}-\mathbb{Z}$. We write $M$ for the quotient by the natural $\mathbb{Z}$-action, and $\overline{M}$ for the partial compactification corresponding to the orbit $\mathbb{Z} \subset \mathbb{C}$, as follows.
\begin{equation*}
\begin{tikzpicture}
	\node (CminusZ) at (0,0) {$\mathbb{C}-\mathbb{Z}$};
	\node (C) at (2,0) {$\mathbb{C}$};
	\node (M) at (0,-1.5) {$M$};
	\node (Mbar) at (2,-1.5) {$\overline{M}$};
	
	\draw[right hook->] (CminusZ) to  (C);
	\draw[right hook->] (M) to  (Mbar);

	\draw[->] (CminusZ) to (M);
	\draw[->] (C) to (Mbar);
\end{tikzpicture}
\end{equation*}
Consequently we may write
\[\overline{M} = \mathbb{P}^1 - \{\text{$2$ points}\} \qquad \text{and} \qquad M = \overline{M} - \{p\}.\]

\begin{remark}For discussion in the physics literature, see for instance work of Aspinwall~\cite[around Figure~2]{Asp}. The point~$p$ is referred to as the conifold point, following physics terminology.
\end{remark}

\begin{remark} As a warning, in the case where $X$ is the local model given by the total space of $\cO_E(-1)^{\oplus 2}$, $M$ may in fact be a double cover of the SKMS, as this $X$ and its flop happen to be isomorphic. For discussion, see joint work of the author and E.~Segal~\cite{DonSeg2} where $M$ appears as the Fayet--Iliopoulos parameter space, in particular \cite[Remark~2.8]{DonSeg2}.\end{remark}

\subsection{Schober on complex plane}\label{subsection schober on C} The contraction $X\to Y$ yields a schober on $(\mathbb{C},i\mathbb{Z})$ with generic fibre $\Dres{X}{E}$, using Proposition~\ref{theorem.standard_flop}. Changing conventions, we may take a schober~$\cP$ on~$(\mathbb{C},\mathbb{Z})$ with monodromy~$\twistFun^{w}$ around the integer~$w+1$, as in Figure~\ref{figure per sheaf eg full}.

\begin{figure}[htb]
\begin{center}
\begin{tikzpicture}
 \node at (0,0) {$\monodromiesCentralCharge{$\Dres{X}{E}$}{$\twistFun^{0}$}{$\twistFun^{-1}$}{$\twistFun^{-2}$}$};
\end{tikzpicture}
\end{center}
\caption{Monodromies for schober $\cP$.}\label{figure per sheaf eg full}
\end{figure}

\begin{notation} Write $\cP^\circ$ for the restriction of $\cP$ to $\mathbb{C}-\mathbb{Z}$.\end{notation}

\subsection{Schober on partial compactification $\overline{M}$} To prepare to construct such a schober, in Theorem~\ref{theorem schober SKMS}, we construct a spherical pair: this is a variant of a result of Bodzenta and Bondal~\cite{BB}, specialized to a standard $3$-fold flop.

\begin{proposition}\label{proposition.standard_flop_pair} Take a standard $3$-fold flopping contraction $X_+\to Y$, with flop $X_-$, and exceptional curves $E_\pm$. There exists a spherical pair $\cS$ determined by
\[ \big\langle \D(X_{-}|E_-)\,,\, \D(\operatorname{pt}) \,\big\rangle = {\cS}_0 = \big\langle \D(X_{+}|E_+)\,,\, \D(\operatorname{pt}) \,\big\rangle\]
inducing the flop functors\, $\mathsf{F}$ of Bondal--Orlov~\cite{BO} as half-monodromies, where ${\cS}_0$ is defined in the proof.

\begin{proof}We show this first for a local model. Consider, as in the proof of Proposition~\ref{theorem.standard_flop}, the GIT quotient $V/\mathbb{C}^*$, with $V = U \oplus U^\vee$ where $U=\mathbb{C}^2$ and the $\mathbb{C}^*$-action is induced by the scalar action of $\mathbb{C}^*$ on $U$. Write $X^{\text{loc}}_\pm$ for the associated quotients. Then Theorem~\ref{theorem.sph_pair} gives a spherical pair 
\begin{equation}\label{equation atiyah pair} \big\langle \D(X^{\text{loc}}_{-})\,,\, \D(\operatorname{pt}) \,\big\rangle = \mathcal{P}_0 = \big\langle \D(X^{\text{loc}}_{+})\,,\, \D(\operatorname{pt}) \,\big\rangle\end{equation}
where we let $w=-1$, so that
\[ \mathcal{P}_0 = \left\{ \begin{array}{c|c} F^{\bullet} \in \D\!\big(\,V/\mathbb{C}^*\,\big) & \operatorname{wt} F^{\bullet} \subseteq [-1,+1] \end{array} \right\}. \]
Half-monodromy equivalences for this spherical pair, from Theorem~\ref{theorem.sph_pair}(\ref{theorem.sph_pair 2}), coincide with the flop functors by, for instance \cite[Proposition~2.3]{DonSeg2}. The categories $\D(\operatorname{pt})$ embed in $\mathcal{P}_0$ by taking $\cO_{\operatorname{pt}}$ to $\cO_{U\! \oplus 0}(-1)$ and $\cO_{0 \oplus U^\vee}(+1)$ respectively, according to Lemma~\ref{lemma.iota_adjoints}, where $(\pm 1)$ denote $\mathbb{C}^*$-weights. Define $\mathcal{S}_0$ as the full subcategory  of $\mathcal{P}_0$ with objects supported on the union of $U\! \oplus 0$ and $0 \oplus U^\vee$. Restricting the semi-orthogonal decompositions from~(\ref{equation atiyah pair}) to $\mathcal{S}_0$ gives, for the local model, a spherical pair with the required properties.
 
Finally, note that for a flop $X_{\pm}$ as in the statement, we have \[\D(X_{\pm}|E_\pm) \cong \D(X^{\text{loc}}_{\pm}|E^{\text{loc}}_\pm)\] by the argument of Proposition~\ref{theorem.standard_flop}, and that these equivalences intertwine the flop functors. The result follows.
\end{proof}
\end{proposition}

\begin{remark} Bodzenta and Bondal~\cite{BB} construct spherical pairs for flops of curves in much greater generality: these spherical pairs are determined by a category with a geometric description, in terms of the birational roof of the flop, which could be adapted to the setting of the proposition above. \end{remark}

The following is standard, and is used below. 

\begin{proposition}\label{proposition comparison} Take a standard $3$-fold flopping contraction $X_+ \to Y$, with flop $X_-$, and exceptional curves $E_\pm$. Then the flop-flop functor $\mathsf{F}\mathsf{F}$ acting on $D(X_\pm)$ and the twist by $\cO_{E_\pm}(-1)$ are inverse.
\end{proposition}
\begin{proof}This follows by work of Toda~\cite[Theorem~3.1]{TodaSph}, noting a sign correction in \cite[Appendix~B]{TodaGV}: these references take $X$ projective, but the methods extend to the quasi-projective case, see for instance~\cite{DW1,DW3}.
\end{proof}

We make the following assumption to prove Theorem~\ref{theorem schober SKMS} below.

\begin{assumption}\label{divisor assumption} Assume there exist line bundles $\cL_\pm$ on $X_\pm$  such that:
\begin{enumerate}
\item\label{divisor assumption 1} the restriction of $\cL_\pm$ to $E_\pm$ is isomorphic to $\cO_{\mathbb{P}^1}(1)$;
\item\label{divisor assumption 2} the restrictions of $\cL_-$ and $\cL_+$ to $X_- -E_- \cong X_+-E_+ $ are dual.
\end{enumerate}
\end{assumption} 

\begin{remark} This assumptions holds in the local model where $X_+$ and its flop $X_-$ are isomorphic to the total space of the bundle $\cO_{\mathbb{P}^1}(-1)^{\oplus 2}$: we may take the pullback of $\cO_{\mathbb{P}^1}(1)$. However, it may fail in general. \end{remark}

We now construct a schober on $\overline{M}$, and relate it to the schober $\cP$ on $\mathbb{C}$  from Subsection~\ref{subsection schober on C}. We let $\overline{M} = \mathbb{P}^1 - \{\pm 1\}$, and define $g\colon \mathbb{C} \to \overline{M}$ taking
\[z \mapsto ({1 - e^{2\pi i z}})/({1 + e^{2\pi i z}})\]
considered as a point of $\mathbb{P}^1$ in the natural way. We let $M = \overline{M} - \{ 0 \}. $

\begin{remark} This $g$ is modelled on the quotient map from Subsection~\ref{sect discussion}.
\end{remark}

\begin{theorem}\label{theorem schober SKMS} Take a standard $3$-fold flopping contraction $X_+\to Y$, with flop $X_-$, and exceptional curves $E_\pm$, satisfying Assumption \ref{divisor assumption}. There exists a schober $\cQ$ on $(\overline{M}, \{0\})$, 
restricting to a local system $\cQ^\circ$ on ${M}$, such that:
\begin{enumerate}
\item\label{theorem schober SKMS 1} the local system  $\cQ^\circ$ is as follows;
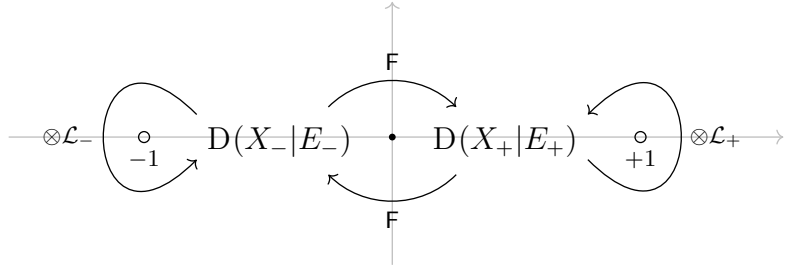
\begin{center}
\begin{tikzpicture}
 \node at (0,0) {$\monodromiesSKMS{$\Dres{X_+}{E_+}$}{$\Dres{X_-}{E_-}$}{$\otimes\cL_+$}{$\otimes\cL_-$}{$\flop$}{1}{0}{$p$}$};
\end{tikzpicture}
\end{center}
\item\label{theorem schober SKMS 3} the data associated to $0$ in $\overline{M}$ is the category ${\mathcal{S}}_0$ of Proposition~\ref{proposition.standard_flop_pair}.
\end{enumerate}
\end{theorem}

\begin{remark} Taking Grothendieck groups gives a local system on~$M$, which should relate to known ways to parallel transport $K$-groups, for instance as in work of Coates, Iritani, and~Jiang~\cite{CIJ}: see in particular~\cite[Figure~3]{CIJ}.
\end{remark}

\begin{proof}[Proof of Theorem~\ref{theorem schober SKMS}]
Take basepoints $x_\pm=g(\pm i)$ in $M$, and choose $\Delta$ in~$\overline{M}$ as below, with $x_\pm$ on its boundary.

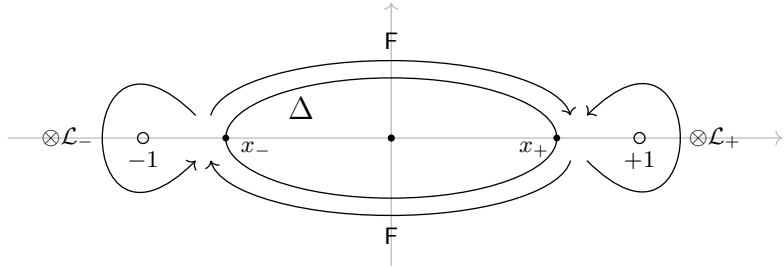
\begin{center}
\begin{tikzpicture}
 \node at (0,0) {$\monodromiesSKMSpair{$x_+$}{$x_-$}{$\otimes\cL_+$}{$\otimes\cL_-$}{$\flop$}{1}{$\Delta$}$};
\end{tikzpicture}
\end{center}

We associate $\Dres{X_\pm}{E_\pm}$ to $x_\pm$, and claim that the diagram of functors above gives a local system $\cQ^\circ$ on $M$. We need only check that the following monodromy around infinity is trivial:
\[ \flop^{-1} \compose \big( \!\placeholder\otimes\,\cL_-\big) \compose \flop^{-1}  \compose \big( \!\placeholder\otimes\,\cL_+\big). \]
This is indeed isomorphic to the identity, by Proposition~\ref{proposition flop and inverse} below, using Assumption \ref{divisor assumption}. We now apply Proposition~\ref{proposition schober pair extend} to extend $\cQ^\circ$ over $0 \in \overline{M}$. We need a spherical pair on $\Delta$ which induces the same local system on $\partial \Delta$ as $\cQ^\circ$: this is supplied by Proposition~\ref{proposition.standard_flop_pair}, and we are done.\end{proof}

The following proposition is standard.

\begin{proposition}\label{proposition flop and inverse} For a standard $3$-fold flop as in  Theorem~\ref{theorem schober SKMS}, satisfying Assumption~\ref{divisor assumption}, we have isomorphic functors $\Dres{X_-}{E_-} \to \Dres{X_+}{E_+}$:
\[ \flop^{-1} \cong \big( \!\placeholder\otimes\,\cL_+^\vee\big) \compose \flop \compose \big( \!\placeholder\otimes\,\cL_-^\vee\big) \] \end{proposition}
\begin{proof}Blowing up $X_\pm$ in $E_\pm$  gives a birational roof $\widehat{X}$ as follows.
\[ X_- \!\xleftarrow{\,\,p_-\,} \widehat{X} \xrightarrow{\,\,p_+\,\,} X_+ \]
To calculate the left-hand side of the claimed isomorphism, we use that $\flop=p^{\phantom{*}}_{-*} \circ p_+^*$, so by taking right adjoints we have \[\flop^{-1}=p_{+*} \circ p_-^! = p_{+*} \circ (\omega_{p_-} \! \otimes \placeholder) \circ p_-^*.\]
Here $\omega_{p_-} \cong \cO_{\widehat{X}}(\widehat{E})$, where $\widehat{E}$ is the shared exceptional locus of $p_\pm$.

Consider now the line bundle
\[ p_-^* \cL_-^\vee  \otimes p_+^* \cL_+^\vee. \]
This is trivial on $\widehat{X}-\widehat{E} \cong X_\pm - E_\pm$ by Assumption~\ref{divisor assumption}(\ref{divisor assumption 2}), and is therefore isomorphic to $\cO_{\widehat{X}}(k \widehat{E})$ for some integer $k$. It restricts to $\widehat{E} \cong \mathbb{P}^1 \times \mathbb{P}^1$ as $\cO(-1,-1)$ by Assumption~\ref{divisor assumption}(\ref{divisor assumption 1}), and the same is true for $\cO_{\widehat{X}}(\widehat{E})$ by an adjunction argument, see for instance \cite[Section~11.3]{HuybrechtsFM}: we deduce~$k=1$.  An isomorphism of functors $\D(X_-) \to \D(X_+) $ as in the statement then follows by the projection formula. Finally, we note that the functors restrict to the categories $\Dres{X_\pm}{E_\pm}$, and the result is proved.
\end{proof}

The following relates $\cQ$ to the schober $\cP$ on $\mathbb{C}$ from Subsection~\ref{subsection schober on C}.

\begin{proposition}\label{proposition schober SKMS} In the setting of Theorem~\ref{theorem schober SKMS}, the pullback $g^*\cQ^\circ$ to $\mathbb{C}-\mathbb{Z}$ is a refinement of $\cP^\circ$  from Subsection~\ref{subsection schober on C}.
\begin{proof}
Observe that $g^{-1}(x_\pm) = \pm i + \mathbb{Z} $, for $x_\pm$ the basepoints from the proof of Theorem~\ref{theorem schober SKMS}. It follows immediately that $g^*\cQ^\circ$ is a local system on $\mathbb{C}-\mathbb{Z}$ given as follows.
\begin{center}
\begin{tikzpicture}
 \node at (0,0) {$\pullbackPic{$\Dres{X_-}{E_-}$}{$\Dres{X_+}{E_+}$}{$\otimes\cL_-$}{$\otimes\cL_+$}{$\mathsf{F}$}{$\mathsf{F}$}$};
\end{tikzpicture}
\end{center}
This is a refinement of $\cP^\circ$ as shown in Figure~\ref{figure per sheaf eg full}, because $\twistFun^{w}$ is inverse to \[\big(\placeholder\otimes\cL_+^{\otimes w+1}\big)\mathsf{F}\mathsf{F}\big(\placeholder\otimes\cL_+^{ \vee \otimes w+1 }\big)\] using Proposition~\ref{proposition comparison}, and so the claim follows.
\end{proof}
\end{proposition}

\subsection{Schober on compactification}\label{subsection compactif} Under the following assumption, the schober~$\cQ$ extends from $\overline{M} = \mathbb{P}^1 - \{\pm 1\}$ to~$\mathbb{P}^1$ in a straightforward manner.

\begin{assumption}\label{strong divisor assumption} Assume there exist  smooth divisors $Z_\pm$ of $X_\pm$  such that:
\begin{enumerate}
\item\label{strong divisor assumption 1} the bundle $ \cO_{X_\pm}(Z_\pm)$ is isomorphic to $\cL_\pm$; 
\item\label{strong divisor assumption 2} the divisors $Z_\pm$ intersects  $E_\pm$ transversely in a single point $x_\pm$.
\end{enumerate}
\end{assumption}  

\begin{proposition} In the setting of Theorem~\ref{theorem schober SKMS}, the schober $\cQ$ extends from $\overline{M}$ to $\mathbb{P}^1$ if Assumption~\ref{strong divisor assumption} holds.

\begin{proof}
We apply Proposition~\ref{proposition schober extend} twice to $\cQ$, taking points $\pm 1$ respectively. For this we need to express, for instance, $\placeholder\otimes\cL_+ \cong \placeholder\otimes\cO_{X_+}(Z_+)$ as a twist of a spherical functor.

We proceed as in Addington~\cite[Section~1.2, Example~4]{Add}. Recall that the cotwist of a functor $\sphFun$ is defined as follows
\begin{equation*}\label{equation cotwist def}\cotwistFun_\sphFun := \operatorname{\sf Cone}\big( \Id \xrightarrow{\text{\,{unit}\,\,}} \sphFun^{\Radj} \compose \sphFun \,\big)\end{equation*}
where $\sphFun^{\Radj}$ denotes a right adjoint. Dropping the subscripts $+$ for brevity, let $i$ be the inclusion of $Z$ in $X$, and take~$\sphFun=i^*$. Then $\sphFun^{\Radj} \compose \sphFun = i_* i^* \cong \placeholder\otimes i_*\cO_Z$ and thence 
\[\cotwistFun_\sphFun = \placeholder\otimes \cO_X(-Z)[1].\] 
Now by Anno--Logvinenko~\cite[Proposition~5.3]{AL2} (noting that the definition of cotwist here differs by a shift) we have a left adjoint
\[\cotwistFun_\sphFun^{\Ladj} \cong \operatorname{\sf Cone}\big( \, \sphFun^{\Ladj} \compose \sphFun \xrightarrow{\text{\,{counit}\,\,}}  \Id\big)[-1] = \twistFun_{\sphFun^{\Ladj}}[-1]. \end{equation*}
Combining, we see that the spherical functor
\[ \sphFun^{\Ladj} = i_! = i_* (\omega_i[\operatorname{dim} i] \otimes \placeholder) \cong i_* (\det \cN_Z X \otimes \placeholder) [-1] \]
from $\D(Z)$ to $\D(X)$ has the required twist, and we are done.
\end{proof}
\end{proposition}


\end{document}